\documentclass[11pt,twoside]{article}
\usepackage{amsmath}
\usepackage{ctex}
\usepackage{amssymb}
\usepackage{fancyhdr}
\usepackage{latexsym}
\usepackage{bbding}
\usepackage{mathrsfs}
\usepackage{wasysym}
\usepackage{cite}
\usepackage{color}
\usepackage{enumerate}
\usepackage[hidelinks]{hyperref}
\usepackage{graphicx}
\usepackage{subfigure}
\usepackage{appendix}
\usepackage{float}
\usepackage{caption}
\usepackage[capitalize]{cleveref}
\usepackage{siunitx}
\usepackage{epstopdf}
\usepackage{indentfirst}
\usepackage{booktabs}
\usepackage[T1]{fontenc}
\usepackage{array}
\usepackage{bookmark}

\allowdisplaybreaks[4]
\newtheorem{theorem}{Theorem}[section]
\newtheorem{lemma}{Lemma}[section]

\newtheorem{definition}[theorem]{Definition}
\newtheorem{remark}{Remark}[section]

\numberwithin{equation}{section}
\newenvironment{proof}[1][Proof]{\noindent\textbf{#1.}}{\hfill$\Box$}
\allowdisplaybreaks

\begin{document}
\date{}
\title{\bf{Propagation phenomena of spatially periodic combustion reaction-diffusion equations around an obstacle}}
\author{Long, Yang $^{a}$,\ \ Jia-Fang Zhang$^{a}$\ \ and\ \ Wei-Jie Sheng$^{b,}$\thanks{Corresponding author (E-mail
address: shengwj09@hit.edu.cn)}\\
\footnotesize{$^{a}$School of Mathematics and Statistics, Henan University,}\\
\footnotesize{Kaifeng, Henan 475001, People's Republic of China}\\
\footnotesize{$^b$School of Mathematics, Harbin Institute of Technology,}\\
\footnotesize{Harbin, Heilongjiang, 150001, People's Republic of China}}
\maketitle

\begin{abstract}
This paper is concerned with propagation phenomena of spatially periodic combustion reaction-diffusion equations in 
exterior domains. It is known that there is a pulsating front connecting 0 and 1 with positive speed in $\mathbb{R}^N$ 
for any direction $e\in\mathbb{S}^{N-1}$. We first prove that there exists an entire solution originating from a pulsating 
front in the exterior domain. Then, we prove that the entire solution propagates completely. Additionally, by constructing 
appropriate super- and sub-solutions, we establish that the entire solution is a transition front connecting 0 and 1, and 
that it is trapped between two translates of the pulsating front as $t\rightarrow +\infty$. Finally, under a suitable 
assumption, we show that the entire solution converges to the same pulsating front as $t\rightarrow 
+\infty$, as well as the uniqueness of such entire solutions.

		
\end{abstract}

\section{Introduction}

In this paper, we investigate the following spatially periodic reaction-diffusion equation around an obstacle
\begin{equation}\label{f1}
\begin{cases}
u_t=\Delta u+f(x,u),\ \ x\in\Omega,\ t\in\mathbb{R},\\
\frac{\partial u}{\partial n}=0,\ \ x\in\partial\Omega,\ t\in\mathbb{R},
\end{cases}
\end{equation}
where $\Omega=\mathbb{R}^N\setminus K$ ($N\geq2$) is called an exterior domain, and $K\subset\mathbb{R}^N$ is a compact set 
with smooth boundary that acts as an obstacle. Moreover, $n(x)$ represents the unit outward normal to $\partial\Omega$ at 
$x$, and the Neumann boundary condition means that there is no flux of $u$ across $\partial\Omega$. The nonlinear term $f$ 
satisfies the following assumptions:
\begin{description}
\item[{\bf(H1)}] $f$ is a function of class $C^\infty(\mathbb{R}^{N+1})$ and for all $k\in\mathbb{N}$, 
$$\Vert f\Vert_{C^k(\mathbb{R}^{N+1})}=\sum_{i=0}^k\Vert D^if \Vert_{L^\infty(\mathbb{R}^{N+1})}<+\infty.$$
\item[{\bf(H2)}] $f$ is $L$-periodic in $x$, that is, there is a vector $L:=(L_1,\cdots,L_N)$ with $L_i>0(i=1,\cdots,N)$ 
such that for any vector $k:=(k_1,\cdots,k_N)\in\mathbb{Z}^N$, $f(x+k L,u)=f(x,u)$, where $kL=(k_1L_1,\cdots,k_NL_N)$.
\item[{\bf(H3)}] There exists $\theta\in(0,1)$ such that
\begin{equation*}
\begin{cases}
\forall(x,u)\in\mathbb{R}^N\times[0,\theta]\cup\{1\},\ f(x,u)=0,\\
\forall(x,u)\in\mathbb{R}^N\times(\theta,1),\ f(x,u) \geq 0,\\
\forall u\in(\theta,1),\ \exists x\in\mathbb{R}^N,\ \text{s.t.}\ f(x,u)>0.
\end{cases}
\end{equation*}
\item[{\bf(H4)}] There holds $\eta:=\sup_{x\in\mathbb{R}^N}f_u(x,1)<0$.
\end{description}
Based on the above assumptions, it is obvious  that there exist $\eta_1<0$ and $0<\theta_1<\min\{\theta, 1-\theta\}$ such 
that
\begin{equation}\label{eta1theta1}
\sup_{(x,u) \in \mathbb{R}^N\times[1-\theta_1,1+\theta_1]} f_u(x,u)\leq \eta_1.
\end{equation}
Additionally, for mathematical convenience, we assume that
\begin{equation*}
f(x,u)=0\ \ \text{for}\ (x,u)\in\mathbb{R}^N\times(-\infty,0)
\end{equation*}
and denote
\begin{equation}\label{M}
M:=\sup_{(x,u)\in\mathbb{R}^N\times\mathbb{R}}\vert f_u(x,u)\vert.
\end{equation}
The reaction term $f(x,u)$ satisfying (H3)-(H4) is called a combustion nonlinearity, and $\theta$ is called the ignition 
temperature.
	
Let us first recall some previous results concerning a class of spatially periodic traveling wave solutions for problem 
\eqref{f1} in the whole space $\mathbb{R}^N$. In other words, we consider
\begin{equation}\label{f2}
u_t=\Delta u+f(x,u),\ \ x \in \mathbb{R}^N, \ t \in \mathbb{R}.
\end{equation}
The equation is closely linked to numerous branches of science, such as population dynamics and combustion kinetics in 
spatially periodic media, and much progress has been made in analyzing its propagation characteristics. To describe its 
rich dynamics, Shigesada et al. \cite{N1} and Xin \cite{J1,J2,J3} introduced the notion of pulsating front. We formulate 
this concept as follows.
\begin{definition}\label{def1}
A pair $(\Phi_e,c_e)$ with $\Phi_e:\mathbb{R}\times\mathbb{R}^N\rightarrow\mathbb{R}$ and $c_e\in\mathbb{R}$ is called a 
pulsating front connecting 0 and 1 of equation \eqref{f2} with effective speed $c_e$ in the direction 
$e\in\mathbb{S}^{N-1}$, if $(\Phi_e,c_e)$ satisfies	the following conditions: 
\begin{description}
\item[(i)]  the function $u(x,t)=\Phi_e(x\cdot e-c_et,x)$ is an entire (classical) solution of equation \eqref{f2};
\item[(ii)]  $\Phi_e(\xi,x)$ is $L$-periodic in $x$; 
\item[(iii)]  $\Phi_e(\xi,x)$ satisfies 
\begin{equation*}
\lim_{\xi\to+\infty}\Phi_e(\xi,x)=0\ \ \text{and}\ \ \lim_{\xi\to-\infty}\Phi_e(\xi,x)=1\ \ 
\text{uniformly for}\ x\in \mathbb
{R}^N.
\end{equation*}
\end{description}
\end{definition}
	
Due to the environmental heterogeneity, both the profile of the pulsating front and its propagation speed generally vary 
with the propagation direction, which is different from the planar front in homogeneous environments \cite{P1}. 
Berestycki and Hamel  \cite{H3} investigated the existence, uniqueness and monotonicity of the pulsating front for 
equation \eqref{f2} and extended their results to general periodic domains. In fact, under assumptions (H1)-(H4), they 
proved that
\begin{itemize}
\item for any direction $e\in\mathbb{S}^{N-1}$, the function $u(x,t)=\Phi_e(x\cdot e-c_et,x)$ is a pulsating front 
connecting 0 and 1 of equation \eqref{f2} with the speed $c_e$;
\item the speed $c_e$ is unique and positive, and the profile $\Phi_e(\xi,x)$ is unique up to a shift in $\xi$;
\item $\Phi_e(\xi,x)$ is strictly decreasing in $\xi$.
\end{itemize}
Alfaro and Giletti \cite{AG} established that the wave speed and wave profile are continuous with respect to the 
propagation direction $e$ for equation \eqref{f2}. Additionally, Bu and He \cite{Z1} investigated the asymptotic behavior 
of pulsating fronts at infinity as they approach the unstable limiting state, while Lyu, Guo and Wang \cite{J4} 
summarized and refined their results, and further analyzed the asymptotic properties of the derivatives of such fronts. 
For the monostable nonlinearity, Berestycki and Hamel \cite{H3} also showed that for each $e\in\mathbb{S}^{N-1}$, there 
exists a unique minimal speed $\overline{c_e}$ such that a pulsating front of equation \eqref{f2} exists if and only if 
the speed $c\geq\overline{c_e}$. Moreover, one can refer to \cite{buchong8,buchong9} for the asymptotic behavior, 
uniqueness and stability of pulsating fronts. For equation \eqref{f2} with a bistable nonlinearity, the existence and 
nonexistence of pulsating fronts are covered in \cite{buchong3,buchong5,buchong2,buchong24} and references therein. More 
discussions on pulsating fronts and their qualitative properties can be found in \cite{buchong11,buchong12,fzhao,
buchong21,buchong22} and references therein. 
	
Coming back to the propagation phenomena of reaction-diffusion equations in the presence of obstacles, Berestycki, Hamel 
and Matano \cite{H1} systematically studied the propagation phenomena of  bistable reaction-diffusion equations in their 
seminal work. To be more precise, they considered problem \eqref{f1} with $f(x,u)$ replaced by $f(u)$, and proved that 
there exists a unique entire solution $u(x,t)$ originating from the planar front $\Phi(x_1-ct)$ of equation \eqref{f2} 
(with $f(x,u)$ replaced by $f(u)$), namely, it holds 
$$u(x,t)-\Phi(x_1-ct)\rightarrow 0 \ \ \text{as}\ t\rightarrow -\infty\ \ \text{uniformly for}\ x\in\overline{\Omega}.$$ 
They further proved that, as $t\rightarrow+\infty$, the entire solution $u(x,t)$ either exhibits complete propagation in 
the sense that 
$$\lim_{t\rightarrow +\infty}u(x,t)=1\ \ \text{locally uniformly for}\  x\in\overline{\Omega}$$ 
or becomes blocked in the sense that 
$$\lim_{t\rightarrow+\infty}u(x,t)=u_\infty(x)\ \ \text{locally uniformly for}\ x\in\overline{\Omega},$$
where $0<u_\infty(x)<1$ for $x \in \overline{\Omega}$. In particular, when the obstacle $K$ is directionally convex with 
respect to a hyperplane or star-shaped, complete propagation occurs, and the entire solution recovers its original 
profile after passing the obstacle, that is
$$u(x,t)-\Phi(x_1-ct)\rightarrow 0 \ \ \text{as}\ t\rightarrow +\infty \ \ \text{uniformly for}\ x\in\overline{\Omega}.$$
Berestycki, Hamel and Matano  \cite{bhm} recently further considered the propagation phenomena of bistable reaction-
diffusion equations	when the obstacle $K$ is a wall of infinite span with many holes. They gave a clear dichotomy 
between	propagation and  blocking for any type of walls of finite thickness. Namely, the traveling front either passes 
through the wall and propagates toward $x_1=+\infty$ (propagation) or is trapped around the wall (blocking). Guo, Hamel 
and Sheng \cite{ghs} showed that there is a unique global mean speed for any transition front connecting 0 and 1, 
regardless of its shape. See also Jia, Wang and Bu \cite{F2}, Sheng et al. \cite{slww} and Li \cite{L1} for time periodic 
equations, Yan and Sheng \cite{ys,ys1} for combustion equations and reaction-diffusion systems, Hoffman, Hupkes and Van 
Vleck \cite{A1} for lattice differential equations, Brasseur and Coville \cite{bc}, Brasseur et al. \cite{bshv}, Qiao, Li 
and Sun \cite{S1} for nonlocal dispersal equations. Guo and Monobe \cite{H2}, Han et al. \cite{HCWY}, Jia, Bao and Bo 
\cite{JBB}, Yan and Sheng \cite{ys2} investigated the almost $V$-shaped fronts of reaction-diffusion equations/systems 
in exterior domains. 
	
Very recently, Jia, Sheng and Wang \cite{F1} firstly studied the propagation phenomena of spatially periodic bistable 
reaction-diffusion equations around an obstacle when the nonlinearity depends on $x\in\mathbb R^N$. Shi and Li \cite{sal} 
investigated the spatially periodic bistable reaction-diffusion equations in exterior domains with \(f(x,u)=f(x_1,u)\). 
Jia et al. \cite{fff} further studied the propagation phenomena of  spatially periodic combustion reaction-diffusion 
equations in exterior domains  when $f(x,u)=f(x_1,u)$. In particular, they proved that, with the aid of a special 
auxiliary function, there is an entire solution $u(x,t)$ originating from the pulsating front  $\Phi_{e_1}(x_1-c_{e_1}t,
x)$, and that it recovers the same profile as $t\rightarrow +\infty$, that is 
$$u(x,t)\rightarrow \Phi_{e_1}(x_1-c_{e_1}t,x)\ \ \text{as}\ t\rightarrow +\infty \ \ \text{uniformly for}\  x\in\overline
{\Omega}.$$
However, the method fails if the reaction term depends on more than one spatial variable. 
Fortunately, inspired by 
\cite{F2,H1,F1,jjj}, we prove that there is an entire solution originating from a pulsating front. Under an additional 
condition, by constructing suitable super- and sub-solutions, we further show that the entire 
solution recovers the same
pulsating front as $t\rightarrow+\infty$ in exterior domains. We show that the entire solution propagates completely and 
establish that it is a transition front connecting 0 and 1. In particular, it is trapped between two translates of the 
pulsating front as $t\rightarrow+\infty$. Also we prove that the entire solution is unique. Throughout this paper, the 
degeneracy of the combustion equation at equilibrium 0 and spatial periodicity in every direction cause difficulties for 
investigating the long-time behavior of the entire solution to the equation.
	
The rest of the paper is structured as follows. In  Section \ref{sec2}, we present some preliminaries and the main 
results. Section \ref{sec3} is devoted to proving the existence of the entire solution. In Section \ref{sec4}, we analyze 
the complete propagation property of the entire solution and its long-time behavior. In Section \ref{sec5}, we study the 
uniqueness of the entire solution.

\section{Preliminaries and main results}\label{sec2}
	
In this section, we first introduce some preliminaries and then present the main results of this paper. By \cite{Z1,J4}, 
one has the following asymptotic behavior of the pulsating fronts for equation \eqref{f2}.
\begin{lemma}\label{lem1}
Let $(\Phi_e,c_e)$ be the pulsating front for equation \eqref{f2}. Then there exist real numbers $C_1,C_2>0$ such that 
\begin{equation*}
\begin{cases}
\Phi_e(\xi,x)\sim C_1e^{-c_e\xi},\ \partial_\xi\Phi_e(\xi,x)\sim -C_1c_ee^{-c_e\xi}\ \ \text{as}\ \xi \rightarrow +\infty,\\
1-\Phi_e(\xi,x)\sim C_2e^{\tau_e\xi}\phi_{\tau_e}(x),\ \partial_\xi\Phi_e(\xi,x)\sim -C_2\tau_ee^{\tau_e\xi}\phi_{\tau_e}
(x)\ \ \text{as}\ \xi\rightarrow -\infty
\end{cases}
\end{equation*}
uniformly for all $x\in\mathbb{R}^N$, where $\tau_e$ is a positive constant depending on $e$, $0<\phi_{\tau_e}(x)\in C^2
(\mathbb{R}^N)$ is $L$-periodic and $\Vert\phi_{\tau_e}(x)\Vert_{L^\infty(\mathbb{R}^N)}=1$.
\end{lemma}
\begin{lemma}\label{lem2}
Let $(\Phi_e,c_e)$ be the pulsating front for equation \eqref{f2}. Then it holds that
$$\lim_{\xi\to +\infty}\frac{\vert\nabla_x\Phi_e(\xi,x)\vert}{\Phi_e(\xi,x)}=0 \ \ \text{uniformly for all}\ x\in\mathbb
{R}^N.$$
\end{lemma}
\begin{lemma}\label{lem3}
Let $(\Phi_e,c_e)$ be the pulsating front for equation \eqref{f2}. Then there exist real  numbers $A,\kappa_1,\kappa_2>0$ 
dependent on the direction $e$ such that
$$\vert\Phi_e(\xi,x)\vert,\vert D\Phi_e(\xi,x)\vert\leq Ae^{-\kappa_1\xi}\ \ \text{for all}\ (\xi,x)\in[0,+\infty) \times 
\mathbb{R}^N,$$	
$$\vert 1-\Phi_e(\xi,x)\vert,\vert D\Phi_e(\xi,x)\vert\leq Ae^{\kappa_2\xi}\ \ \text{for all}\ (\xi,x)\in(-\infty,0] 
\times\mathbb{R}^N,$$	
where $D$ denotes any first-order derivative with respect to $(\xi,x)\in\mathbb{R}\times\mathbb{R}^N$.
\end{lemma}	
	
To streamline the discussion, we take $e=e_1=(1,0,\cdots,0)$ without loss of generality. Then $\xi=x_1-c_{e_1}t$, where 
$x_1$ is the first component of $x$. Note that if $(\Phi_{e_1},c_{e_1})$ is a pulsating front connecting 0 and 1 for 
equation \eqref{f2} in the direction $e_1$, then it holds that
\begin{equation}\label{f3}
-c_{e_1}\partial_\xi\Phi_{e_1}-\partial_{\xi\xi}\Phi_{e_1}-\Delta_x\Phi_{e_1}-2\partial_{x_1}\partial_\xi\Phi_{e_1}=f(x, 
\Phi_{e_1})\ \ \text{for all}\ (\xi,x)\in\mathbb{R}\times\mathbb{R}^N.
\end{equation}
Since $\Phi_{e_1}(\xi,x)$ is strictly decreasing in $\xi$, one gets that there exist real numbers $C,M_0>0$ such that
\begin{equation}\label{CM0}
\begin{cases}
\Phi_{e_1}(\xi,x)\leq \frac{1}{2}\theta_1,\ \ \text{when}\ \xi\geq C,\\
\Phi_{e_1}(\xi,x)\geq 1-\frac{1}{2}\theta_1,\ \ \text{when}\ \xi\leq -C,\\
\partial_\xi\Phi_{e_1}(\xi,x)\leq -M_0\ \ \text{when}\ -C\leq\xi\leq C,
\end{cases}
\ \ \ \ \text{uniformly for}\ x\in \mathbb{R}^N.
\end{equation}
It can be inferred from Lemmas \ref{lem1}-\ref{lem3} that
\begin{equation}\label{ddd}
\begin{cases}
\lim_{\xi \to +\infty} \frac{\partial_\xi\Phi_{e_1}(\xi,x)}{\Phi_{e_1}(\xi,x)}=-c_{e_1},  \ \lim_{\xi \to -\infty} \frac{
\partial_\xi\Phi_{e_1}(\xi,x)}{\Phi_{e_1}(\xi,x)}=0,\\
\lim_{\xi \to +\infty} \frac{\vert\nabla_x\Phi_{e_1}(\xi,x)\vert}{\Phi_{e_1}(\xi,x)}=0,\ \lim_{\xi \to -\infty} \frac{
\vert\nabla_x\Phi_{e_1}(\xi,x)\vert}{\Phi_{e_1}(\xi,x)}=0
\end{cases}
\ \ \ \ \text{uniformly for}\ x\in \mathbb{R}^N.
\end{equation}
Define
\begin{equation}\label{M1M2}
\begin{split}
M_1:=\left\Vert\frac{\partial_\xi\Phi_{e_1}(\xi,x)}{\Phi_{e_1}(\xi,x)}\right\Vert_{L^\infty(\mathbb{R}^{N+1})}<+\infty\ 
\ \text{and}\ \ M_2:=\left\Vert\frac{\vert\nabla_x\Phi_{e_1}(\xi,x)\vert}{\Phi_{e_1}(\xi,x)}\right\Vert_{L^\infty(\mathbb
{R}^{N+1})}<+\infty.
\end{split}
\end{equation}
Moreover, by \eqref{ddd} there exist $C_3>C$ and $M_3>0$ such that
\begin{equation}\label{C3M3}
\begin{cases}
\frac{c_{e_1}}{2}\leq\frac{\vert\partial_\xi\Phi_{e_1}(\xi,x)\vert}{\Phi_{e_1}(\xi,x)}\leq \frac{3c_{e_1}}{2},\ \text
{when}\ \xi\geq C_3,\\
\frac{\vert\nabla_x\Phi_{e_1}(\xi,x)\vert}{\Phi_{e_1}(\xi,x)}\leq\frac{c_{e_1}}{16},\ \text{when}\ \xi\geq C_3,\\
\partial_\xi\Phi_{e_1}(\xi,x)\leq-M_3,\ \text{when}\ -C_3\leq\xi\leq C_3
\end{cases}
\ \ \ \ \text{uniformly for}\ x\in \mathbb{R}^N. 
\end{equation}
Besides, Lemma \ref{lem3} also implies that there exist real numbers $A_1, \alpha_1, \alpha_2>0$ such that
\begin{equation}\label{A1}
\begin{split}
\Phi_{e_1}(\xi,x)&\leq A_1e^{-\alpha_1\xi}\ \ \text{for all}\ (\xi,x)\in[0,+\infty) \times \mathbb{R}^N,\\
\vert\partial_\xi\Phi_{e_1}(\xi,x)\vert,\vert \nabla_x&\Phi_{e_1}(\xi,x)\vert\leq A_1e^{\alpha_2\xi}\ \  \text{for all}\ 
(\xi,x)\in(-\infty,0] \times \mathbb{R}^N.
\end{split}
\end{equation}
	
Now,  let us give our main results. We first show that there is an entire solution of problem \eqref{f1} originating from 
the pulsating front $\Phi_{e_1}(x_1-c_{e_1}t,x)$.
\begin{theorem}\label{th2}
There exists an entire solution $u(x,t)$ of problem \eqref{f1} satisfying
\begin{equation}\label{uxt1}
\lim_{t\to -\infty}\frac{u(x,t)-\Phi_{e_1}(x_1-c_{e_1}t,x)}{\Phi_{e_1}^\beta(x_1-c_{e_1}t,x)}=0 \ \ \text{uniformly for}\ 
x \in \overline{\Omega} \ \text{and}\  \beta \in \left[0,\beta_0\right],
\end{equation}
where $0<\beta_0<1$ is a constant. Moreover, $0< u(x,t) <1 $ and $u_t(x,t)>0$ for all $x \in \overline{\Omega}$ and $t 
\in \mathbb{R}$.
\end{theorem}
	
Then we prove that the entire solution $u(x,t)$ given in Theorem  \ref{th2} propagates completely.
\begin{theorem}\label{th3}
Let the entire solution $u(x,t)$ be given by Theorem \ref{th2}. Then there holds
\begin{equation*}
\lim_{t \rightarrow +\infty}u(x,t)=1 \ \ \text{locally uniformly for}\  x \in \overline{\Omega}.
\end{equation*}
\end{theorem}
	
Next, we show that the entire solution $u(x,t)$ is a transition front connecting $0$ and $1$ with global mean speed 
$c_{e_1}$. To this end, let us first recall the definitions of transition fronts and their global mean speed. The 
terminology comes from the pioneering works of Berestycki and Hamel \cite{bh1,jjj}. See also Shen  \cite{shen1} for the 
one-dimensional work. For any two points $x,y \in \overline{\Omega}$, let $d_{\Omega}(x,y)$ denote the infimum of the arc 
lengths of all $C^1$ curves joining $x$ to $y$ in $\overline{\Omega}$ (i.e., the geodesic distance between $x$ and $y$ in 
$\overline{\Omega}$). Given any subsets $A, B\subset \overline{\Omega}$ and a point $x\in\overline{\Omega}$, we set
$$d_{\Omega}(A,B)=\inf\{d_{\Omega}(x,y):(x,y) \in A \times B \} \ \ \text{and}\ \ d_{\Omega}(x,A)=d (\{x\},A).$$ 
Consider two families of open nonempty subsets $\{\Omega_{t}^{-}\}_{t\in \mathbb{R}}$ and $\{\Omega_{t}^ { + } \}_{t\in 
\mathbb{R}} $ of $\Omega$ satisfying for any $t\in\mathbb{R}$,
\begin{equation}\label{subset1}
\left\{\begin{array}{l}\Omega_t^-\cap\Omega_t^+=\emptyset,\\
\partial\Omega_t^-\cap\Omega=\partial\Omega_t^+\cap\Omega=:\Gamma_t,\\
\Omega_t^-\cup\Gamma_t\cup\Omega_t^+=\Omega,\\
\sup\{d_\Omega(x,\Gamma_t):x\in\Omega_t^-\}=\sup\{d_\Omega(x,\Gamma_t):x\in\Omega_t^+\}=+\infty
\end{array}\right.
\end{equation}
and
\begin{equation}\label{subset2}
\begin{cases}
\inf\left\{\sup\{d_{\Omega} (y,\Gamma_t):y\in\Omega_t^+,d (x,y)\leq r\}:t\in\mathbb{R},x\in\Gamma_t\right\}\to+\infty,\\
\inf\left\{\sup\{d_{\Omega} (y,\Gamma_t):y\in\Omega_t^-,d (x,y)\leq r\}:t\in\mathbb{R},x\in\Gamma_t\right\}\to+\infty
\end{cases}
\ \ \ \ \text{as}\ r\to+\infty.
\end{equation}
Notice that the condition   \eqref{subset2} implies that for any $M>0$, there exists a constant $r_M>0$ such that for 
each $t\in\mathbb R$ and $x\in\Gamma_t$, there are $y^\pm\in\Omega^\pm_t$ such that
\begin{align*}
d_\Omega(x,y^\pm)\leq r_M\ \ \text{and}\ \ d_\Omega(y^\pm,\Gamma_t)\geq M.
\end{align*}
Moreover, the sets $\Gamma_t$ are assumed to be included in a finite number of graphs. Namely, there is an integer $n\geq
1$ such that for each $t\in\mathbb R$, there exist $n$ open subsets $\omega_{i,t}\subset\mathbb{R}^{N-1}$ ($1\leq i\leq 
n$), $n$ continuous maps $\psi_{i,t}:\omega_{i,t}\to\mathbb{R}$ and $n$ rotations $R_{i,t}$ of $\mathbb{R}^N$, with
\begin{equation}\label{subset3}
\Gamma_{t}\subset\bigcup_{1\leq i\leq n}R_{i,t}\left(\left\{x\in\mathbb{R}^{N}:x^{\prime}\in\omega_{i,t},x_{N}=\psi_{i} 
\left( x ^ { \prime } \right) \right\} \right).
\end{equation}
\begin{definition}\label{tran}
A classical solution $u: \overline\Omega\times \mathbb{R} \to (0,1)$ of problem \eqref{f1} is called a transition front 
connecting 0 and 1, if there exist some subsets $\{\Omega_t^\pm\}_{t\in\mathbb{R}}$ and 	$\{\Gamma_t\}_{t\in\mathbb{R}
}$ satisfying \eqref{subset1}, \eqref{subset2} and\eqref{subset3} such that for any $\varepsilon>0$, there is a positive 
constant $M_\varepsilon>0$ satisfying
\begin{equation*}\label{tran1}
\begin{cases}
\forall t\in\mathbb R,\ \forall x\in\overline{\Omega_t^+},\ \ d_{\Omega} ( x,\Gamma_t)\geq M_\epsilon\implies u(x,t)\geq 
1-\epsilon,\\
\forall t\in\mathbb R,\ \forall x\in\overline{\Omega_t^-},\ \ d_{\Omega} (x,\Gamma_t)\geq M_\epsilon \implies u(x,t)\leq
\epsilon.
\end{cases}
\end{equation*}
Moreover, $u$ is said to admit a global mean speed $\gamma\ (\geq0)$ if
\begin{equation*}\label{g-speed}
\frac{d_{\Omega}(\Gamma_t,\Gamma_s)}{|t-s|}\to\gamma\ \ \text{as}\ \vert t-s\vert\to +\infty.
\end{equation*}
\end{definition}
Great attention has been paid to the study of transition fronts since the pioneering works of \cite{bh1,jjj}; see also \cite
{hr,hr1,hamel,bo,shen1,shen2,shen3,shen4} and references therein. We show that the entire solution given in Theorem \ref
{th2} is a transition front connecting 0 and 1 in the sense of Definition \ref{tran}. Furthermore, this solution is 
trapped between two translates of the pulsating front as $t\rightarrow +\infty$.
\begin{theorem}\label{th4}
Let the entire solution $u(x,t)$ be given by Theorem \ref{th2}. Then $u(x,t)$ is a transition front connecting $0$ and 
$1$ with global mean speed $c_{e_1}$, and it satisfies
$$\sup_{\substack{(x,t)\in\overline{\Omega}\times\mathbb{R} \\ x_1 - c_{e_1} t \leq -\mathcal{R}}}(1-u(x,t)) \rightarrow 
0 \ \ \text{and}\ \ \sup_{\substack{(x,t)\in\overline{\Omega}\times\mathbb{R} \\ x_1 - c_{e_1}t\geq\mathcal{R}}}u(x,t) 
\rightarrow 0 \ \ \text{as}\ \mathcal{R} \rightarrow +\infty.$$
Moreover, for any $\beta \in (0,\min\{\beta_0,\beta_1\}]$, there exist two real numbers $\tau_1\leq 0\leq\tau_2$ such that
$$\limsup_{t\to +\infty}\sup_{x\in\overline{\Omega}}\frac{u(x,t)-\Phi_{e_1}(x_1-c_{e_1}t+\tau_1,x)}{\Phi_{e_1}^\beta(x_1-
c_{e_1}t+\tau_1, x)}\leq0$$
and
$$\liminf_{t\to +\infty}\inf_{x\in\overline{\Omega}}\frac{u(x,t)-\Phi_{e_1}(x_1-c_{e_1}t+\tau_2,x)}{\Phi_{e_1}^\beta(x_1-
c_{e_1}t+\tau_2, x)}\geq 0,$$
where $\beta_0$ is given by Theorem \ref{th2} and $0<\beta_1<1$ is a constant.
\end{theorem}
	
The following two theorems are proved under a suitable assumption \eqref{jkl}. Theorem \ref{th5} states that the entire 
solution $u(x,t)$ can recover its profile like the pulsating front $\Phi_{e_1}(x_1-c_{e_1}t,x)$, and Theorem \ref{th6} 
establishes the uniqueness of the entire solution.
\begin{theorem}\label{th5}
Let $u(x,t)$ be given by Theorem \ref{th2}. If $u(x,t)$ further satisfies  
\begin{equation}\label{jkl}
\sup_{\substack{(x,t)\in\overline{\Omega}\times\mathbb{R} \\ x_1 - c_{e_1} t \geq \mathcal{R}}} \left\vert\frac{u(x,t)}
{\Phi_{e_1}(x_1-c_{e_1}t,x)}-1\right\vert\rightarrow 0 \ \ \text{as}\ \mathcal{R}\rightarrow  +\infty.
\end{equation}
Then, for any $\beta \in (0,\min\{\beta_0,\beta_1\}]$, there holds
\begin{equation*}
\lim_{t \to +\infty} \frac{u(x,t)-\Phi_{e_1}(x_1-c_{e_1}t,x)}{\Phi_{e_1}^\beta(x_1-c_{e_1}t,x)}=0 \ \ \text{uniformly 
for}\ x \in \overline{\Omega},
\end{equation*}
where $\beta_0$ and $\beta_1$ are defined in Theorem \ref{th4}.
\end{theorem}
\begin{remark}{\rm
Here we adopt the assumption \eqref{jkl} from \cite{F1} to prove the conclusion. Actually, if the disturbance generated 
by the obstacle on the propagation process of the pulsating front decays exponentially with a sufficiently large decay 
rate $\lambda>c_{e_1}$ as $x_1-c_{e_1}t\rightarrow  +\infty$, then the assumption \eqref{jkl} holds true. It is of great 
interest and challenge to weaken or remove the assumption, and we leave it for further studies.}
\end{remark}
\begin{theorem}\label{th6}
Let the entire solution $u(x,t)$ be given by Theorem \ref{th2}. If there exists another entire solution $\widehat{u}(x,t)$ 
satisfying the conditions of Theorem \ref{th2}, and both $u(x,t)$ and $\widehat{u}(x,t)$ satisfy the assumption 
\eqref{jkl}, then $u(x,t)\equiv\widehat{u}(x,t)$ for all $(x,t)\in \overline{\Omega}\times \mathbb{R}$.
\end{theorem}
\begin{remark}
{\rm 
In fact, we show that Theorem \ref{th6} holds under a weaker condition: that is
$u(x,t)$ and $\widehat{u} (x,t)$ satisfy
$$u(x,t)-\Phi_{e_1}(x_1-c_{e_1}t,x)\rightarrow 0\ \  \text{as}\  t\rightarrow +\infty\ 
 \text{uniformly for}\ x\in \overline{\Omega}$$
and
$$\widehat{u}(x,t)-\Phi_{e_1}(x_1-c_{e_1}t,x)\rightarrow 0\ \  \text{as}\  
t\rightarrow +\infty\  \text{uniformly for} 
\ x\in \overline{\Omega}.$$
}
\end{remark}

\section{Existence of entire solutions}\label{sec3}

In this section, we first construct super- and sub-solutions of problem \eqref{f1} for very negative time. Then we prove 
that there exists an entire solution originating from a pulsating front. Let us first construct a non-negative function 
$\delta_ 0(x)\in C^2(\mathbb{R}^N)$ by using the classical distance function from \cite{D1}. The function $\delta_ 0(x)$ 
can be chosen such that $\nabla \delta_ 0(x) \cdot n(x)=1$ for $x \in \partial\Omega$, and its support is contained in 
$B(0,r_1)$ with $r_1>r>0$ $(K\subset B(0,r))$. Define $\delta(x)=\delta_0(x)+C_\delta$, where $C_\delta>0$ is a 
sufficiently large constant such that $\delta(x)$ satisfies
\begin{equation}\label{delta}
\begin{split}
\delta(x)&\geq 1, \\
\left\Vert\frac{\vert\nabla\delta(x)\vert}{\delta(x)}\right\Vert_{L^{\infty}(\mathbb{R}^N)}&\leq\frac{c_{e_1}\alpha_1}{4( 
1+M_1+M_2)}, \\
\left\Vert\frac{\Delta\delta(x)}{\delta(x)}\right\Vert_{L^{\infty}(\mathbb{R}^N)} &\leq\frac{c_{e_1}\alpha_1}{4(1+M_1+M_2
)},
\end{split}
\end{equation}
where $\alpha_1>0$ is given by \eqref{A1} and $M_1,M_2>0$ are given by \eqref{M1M2}.
		
\begin{lemma}\label{lem5}
Let $\beta_0=\min\left\{\frac{1}{2},\frac{c_{e_1}\alpha_1}{8M_1M_2},\frac{1}{8c_{e_1}\Vert\delta(x)\Vert_{L^\infty(\mathbb
{R}^N)}}\right\}$. Then there exist $\rho>0$, $\zeta>0$ and $T<0$ such that the following functions
$$u^+(x,t)=\Phi_{e_1}(\xi^+(x,t),x)+\zeta\Phi_{e_1}^{\beta_0}(\xi,x)\delta(x)e^{\frac{1}{2}c_{e_1}\alpha_1t}$$
and 
$$u^-(x,t)=\Phi_{e_1}(\xi^-(x,t),x)-\zeta\Phi_{e_1}^{\beta_0}(\xi,x)\delta(x)e^{\frac{1}{2}c_{e_1}\alpha_1t}$$
are the super-solution and sub-solution of problem \eqref{f1} for $x\in \overline{\Omega}$ and $t\leq T$, respectively, where 
$$\xi^+(x,t)=x_1-c_{e_1}(t+\rho e^{\frac{1}{2}c_{e_1}\alpha_1t})\ \ \text{and}\ \ 
\xi^-(x,t)=x_1-c_{e_1}(t-\rho e^{\frac
{1}{2}c_{e_1}\alpha_1t}),$$
$\alpha_1>0$ is given by \eqref{A1} and $M_1,M_2>0$ are given by \eqref{M1M2}.
\end{lemma}
		
\begin{proof} 
{\it Step 1: choice of parameters.} Take 
$$\zeta\geq 8c_{e_1}(1+A_1)e^{\alpha_1 r}, \rho\geq\frac{2M\zeta\Vert\delta(x)\Vert_{L^\infty(\mathbb{R}^N)}}{M_0c_{e_1}^2
\alpha_1}$$
and
$$T\leq\min\left\{-\frac{r}{c_{e_1}},\ T_1,\ \frac{2}{c_{e_1}\alpha_1}\ln {\frac{\theta_1}{2\zeta\Vert\delta(x)\Vert_{L^
\infty(\mathbb{R}^N)}}}\right\}$$
in that order, where $A_1,\alpha_1>0$ are given by \eqref{A1}, $M, M_0>0$ are given by \eqref{M} and \eqref{CM0}, and 
$\theta_1$ is given by \eqref{eta1theta1}.

Here we provide details on the choice of $T$. It follows from $T\leq -\frac{r}{c_{e_1}}$ that $-r-c_{e_1}T\geq 0$. Since 
$K \subset B(0,r)$, we can get that $\xi(x,T)\geq 0$ for all $x \in \partial\Omega $. In this case, from \eqref{A1} we 
have that
$$\Phi_{e_1}(\xi,x)\leq A_1e^{-\alpha_1\xi} \ \ \text{for all}\  x \in \partial\Omega\ \ \text{and}\ t\leq T.$$
Additionally, by \eqref{ddd} and Lemma \ref{lem1}, it follows that
$$\frac{\partial_\xi\Phi_{e_1}(\xi,x)}{\Phi_{e_1}(\xi,x)}\rightarrow -c_{e_1} 
\ \ \text{as}\  t\rightarrow -\infty \ 
\text{uniformly for}\  x \in \partial\Omega \subset B(0,r), $$
$$\frac{\vert\nabla_x\Phi_{e_1}(\xi,x)\vert}{\Phi_{e_1}(\xi,x)}\rightarrow 0 \ \ \text{as}\  
t\rightarrow -\infty \ 
\text{uniformly for}\   x \in \partial\Omega \subset B(0,r), $$
$$\frac{\Phi_{e_1}(\xi^+,x)}{\Phi_{e_1}(\xi,x)}\sim e^{c_{e_1}^2\rho e^{\frac{1}{2}c_{e_1}\alpha_1t}} \ \ \text{as}\ 
t\rightarrow -\infty \ \text{uniformly for}\ x \in \partial\Omega \subset B(0,r). $$
Therefore, there exists a sufficiently negative constant $T_1<0$ dependent on $\rho$ but independent of $\zeta$, such 
that for $T\leq T_1$, there holds
$$\frac{\vert\partial_\xi\Phi_{e_1}(\xi^+,x)\vert}{\Phi_{e_1}(\xi,x)},\ 
\frac{\vert\nabla_x\Phi_{e_1}(\xi^+,x)\vert}{\Phi_{e_1}(\xi,x)},\  
\frac{\vert\nabla_x\Phi_{e_1}(\xi,x)\vert}{\Phi_{e_1}(\xi,x)},\   
\frac{\vert\partial_\xi\Phi_{e_1}(\xi,x)\vert}{\Phi_{e_1}(\xi,x)} \leq 2c_{e_1}$$
for all $x \in \partial\Omega$ and $t\leq T$. By a similar argument, the same estimates hold when replacing $\xi^+(x,t)$ 
with $\xi^-(x,t)$, even if this requires choosing a more negative $T_1$.

{\it Step 2: verification of the differential inequalities.} Define
\begin{equation*}
\mathcal{L} [u](x,t)=u_t-\Delta u-f(x,u).
\end{equation*}
We only prove that $\mathcal{L} [u^+](x,t)\geq 0$ for $x\in \Omega$ and $t\leq T$, since $\mathcal{L} [u^-](x,t)\leq 0$ 
can be treated similarly. Since $\beta_0\leq \min\left\{\frac{1}{2},\frac{c_{e_1}\alpha_1}{8M_1M_2}\right\}$, using 
\eqref{f3}, \eqref{M1M2} and \eqref{delta}, after some calculations one infers that
\begin{align*}
\mathcal{L}[u^+](x,t)=&\ f(x,\Phi_{e_1}(\xi^+,x))-f(x,u^+(x,t))-\frac{1}{2}c_{e_1}^2\rho\alpha_1\partial_\xi\Phi_{e_1}(\xi
^+,x)e^{\frac{1}{2}c_{e_1}\alpha_1t}\\
&\ +\frac{1}{2}\zeta c_{e_1}\alpha_1\Phi_{e_1}^{\beta_0}(\xi,x)\delta(x)e^{\frac{1}{2}c_{e_1}\alpha_1t}+\zeta\beta_0\Phi_
{e_1}^{\beta_0-1}(\xi,x)\delta (x)e^{\frac{1}{2}c_{e_1}\alpha_1 t}f(x,\Phi_{e_1}(\xi,x))\\
&\ +\zeta\beta_0(1-\beta_0)\Phi_{e_1}^{\beta_0-2}(\xi,x)\delta (x)e^{\frac{1}{2}c_{e_1}\alpha_1 t}\Big[ \vert\partial_\xi
\Phi_{e_1}(\xi,x)\vert^2+\vert\nabla_x\Phi_{e_1}(\xi,x)\vert^2\\
&\ +2\partial_\xi\Phi_{e_1}(\xi,x)\partial_{x_1}\Phi_{e_1}(\xi,x)\Big]-\zeta\Phi_{e_1}^{\beta_0}(\xi,x)\Delta\delta(x)e^{ 
\frac{1}{2}c_{e_1}\alpha_1t}\\
&\ -2\zeta\beta_0\Phi_{e_1}^{\beta_0-1}(\xi,x)\Big[\partial_\xi\Phi_{e_1}(\xi,x)\partial_{x_1}\delta(x)+\nabla_x\Phi_{e_1}
(\xi,x)\cdot \nabla\delta(x)\Big]e^{\frac{1}{2}c_{e_1}\alpha_1 t}\\
\geq&\ f(x,\Phi_{e_1}(\xi^+,x))-f(x,u^+(x,t))-\frac{1}{2}c_{e_1}^2\rho\alpha_1\partial_\xi\Phi_{e_1}(\xi^+,x)e^{\frac{1}
{2}c_{e_1}\alpha_1t}\\
&\ +\zeta\Phi_{e_1}^{\beta_0}(\xi,x)\delta(x)e^{\frac{1}{2}c_{e_1}\alpha_1t}\bigg[\frac{1}{2}c_{e_1}\alpha_1-2\beta_0(1- 
\beta_0) \frac{\vert\partial_\xi\Phi_{e_1}(\xi,x)\vert}{\Phi_{e_1}(\xi,x)} \frac{\vert\nabla_x\Phi_{e_1}(\xi,x)\vert}
{\Phi_{e_1}(\xi,x)}\\
&\ -\frac{\vert\Delta\delta(x)\vert}{\delta(x)}-2\beta_0 \frac{\vert\nabla\delta(x)\vert}{\delta(x)} \frac{\vert\partial_
\xi\Phi_{e_1}(\xi,x)\vert}{\Phi_{e_1}(\xi,x)}-2\beta_0 \frac{\vert\nabla_x\Phi_{e_1}(\xi,x)\vert}{\Phi_{e_1}(\xi,x)} \frac
{\vert\nabla\delta(x)\vert}{\delta(x)} \bigg]\\
\geq&\ f(x,\Phi_{e_1}(\xi^+,x))-f(x,u^+(x,t))-\frac{1}{2}c_{e_1}^2\rho\alpha_1\partial_\xi\Phi_{e_1}(\xi^+,x)e^{\frac{1}
{2}c_{e_1}\alpha_1t}.
\end{align*}
We consider the three cases in the sequel: $ \xi^+(x,t)\geq C$, $\xi^+(x,t)\leq -C$ and $-C\leq\xi^+(x,t)\leq C$, where 
$C$ is given by \eqref{CM0}.
	
{\bf Case 1: $\xi^+(x,t)\geq C$.} It follows from \eqref{CM0} and  $t\leq T\leq\frac{2}{c_{e_1}\alpha_1}\ln {\frac
{\theta_1}{2\zeta\Vert\delta(x)\Vert_{L^\infty(\mathbb{R}^N)}}}$ that
$$0<\Phi_{e_1}(\xi^+,x)<u^+(x,t)\leq \theta_1,$$
so we get that $f(x,\Phi_{e_1}(\xi^+,x))=f(x,u^+(x,t))=0$ by \eqref{eta1theta1}. Since $\partial_\xi\Phi_{e_1}(\xi^+,x)
<0$, we have $\mathcal{L} [u^+](x,t)\geq 0$.
	
{\bf Case 2: $\xi^+(x,t)\leq -C$.} It follows from \eqref{CM0} and  $t\leq T\leq \frac{2}{c_{e_1}\alpha_1}\ln {\frac
{\theta_1}{2\zeta\Vert\delta(x)\Vert_{L^\infty(\mathbb{R}^N)}}}$ that
$$1-\frac{1}{2}\theta_1\leq\Phi_{e_1}(\xi^+,x))<u^+(x,t)\leq 1+\frac{1}{2}\theta_1.$$
Since $f_u(x,u)\leq \eta_1<0$ for $(x,u)\in\mathbb{R}^N\times[1-\theta_1,1+\theta_1]$ and $\partial_\xi\Phi_{e_1}(\xi^+,x)
<0$, it follows that $\mathcal{L} [u^+](x,t)\geq 0$.

{\bf Case 3: $ -C\leq \xi^+(x,t)\leq C $.}  the mean value theorem yields
$$f(x,\Phi_{e_1}(\xi^+,x))-f(x,u^+(x,t))\geq-M\zeta\Vert\delta(x)\Vert_{L^\infty(\mathbb{R}^N)}e^{\frac{1}{2}c_{e_1}\alpha
_1t},$$
where $M>0$ is given by \eqref{M}. Then, by \eqref{CM0} and the choice of $\rho$ we can conclude that there holds
$\mathcal{L} [u^+](x,t)\geq 0$.
		
{\it Step 3: verification of the boundary condition.} Some computations imply that  
\begin{equation*}
\begin{split}
\nabla_x u^+(x,t) \cdot n(x)=&\ \partial_\xi\Phi_{e_1}(\xi^+,x)n_1(x)+\nabla_x\Phi_{e_1}(\xi^+,x)\cdot n(x)\\
&\ +\zeta\beta_0\Phi_{e_1}^{\beta_0-1}(\xi,x)\delta(x)e^{\frac{1}{2}c_{e_1}\alpha_1t}[\partial_\xi\Phi_{e_1}(\xi,x) n_1
(x)\\
&\ +\nabla_x\Phi_{e_1}(\xi,x)\cdot n(x)]+\zeta\Phi_{e_1}^{\beta_0}(\xi,x)e^{\frac{1}{2}c_{e_1}\alpha_1t}\nabla\delta(x)
\cdot n(x),
\end{split}
\end{equation*}
where $n(x)$ represents the unit outward normal to $\partial\Omega$ at $x$, and $n_1(x)$ denotes the first component of 
$n(x)$.	
By the choice of $T$ and $\beta_0\leq 
\min\left\{\frac{1}{2},\frac{1}{8c_{e_1}\Vert\delta(x)\Vert_{L^\infty(\mathbb{R}^N)}}\right\}$, one infers that for all 
$x \in \partial\Omega$ and $t\leq T$,
\begin{align*}
\nabla_x u^+(x,t)\cdot n(x)\geq&\ -\frac{\vert\partial_\xi\Phi_{e_1}(\xi^+,x)\vert}{\Phi_{e_1}(\xi,x)}\Phi_{e_1}^{1-
\beta_0}(\xi,x)\Phi_{e_1}^{\beta_0}(\xi,x)-\frac{\vert\nabla_x\Phi_{e_1}(\xi^+,x)\vert}{\Phi_{e_1}(\xi,x)}\Phi_{e_1}^{1-
\beta_0}( \xi,x)\times\\
&\ \Phi_{e_1}^{\beta_0}(\xi,x)-\zeta\beta_0\Vert\delta(x)\Vert_{L^\infty(\mathbb{R}^N)}\frac{\vert\partial_\xi\Phi_{e_1}
(\xi,x)\vert}{\Phi_{e_1}(\xi,x)}\Phi_{e_1}^{\beta_0}(\xi,x)e^{\frac{1}{2}c_{e_1}\alpha_1t}\\
&\ -\zeta\beta_0\Vert\delta(x)\Vert_{L^\infty(\mathbb{R}^N)}\frac{\vert\nabla_x\Phi_{e_1}(\xi,x)\vert}{\Phi_{e_1}(\xi,x)}
\Phi_{e_1}^{\beta_0}(\xi,x)e^{\frac{1}{2}c_{e_1}\alpha_1t}\\
&\ +\zeta\Phi_{e_1}^{\beta_0}(\xi,x)e^{\frac{1}{2}c_{e_1}\alpha_1t}\\
\geq &\ \zeta\Phi_{e_1}^{\beta_0}(\xi,x)e^{\frac{1}{2}c_{e_1}\alpha_1t}-4c_{e_1}A_1^{1-\beta_0}e^{-\alpha_1(1-\beta_0)\xi}
\Phi_{e_1}^{\beta_0}(\xi,x)\\
&\ -4c_{e_1}\zeta\beta_0\Vert\delta(x)\Vert_{L^\infty(\mathbb{R}^N)}\Phi_{e_1}^{\beta_0}(\xi,x)e^{\frac{1}{2}c_{e_1}\alpha
_1t}\\
\geq &\  \zeta\Phi_{e_1}^{\beta_0}(\xi,x)e^{\frac{1}{2}c_{e_1}\alpha_1t}-4c_{e_1}A_1^{1-\beta_0}e^{\alpha_1(1-\beta_0)r}e^
{c_{e_1}\alpha_1(1-\beta_0)t}\Phi_{e_1}^{\beta_0}(\xi,x)\\
&\ -\frac{1}{2}\zeta\Phi_{e_1}^{\beta_0}(\xi,x)e^{\frac{1}{2}c_{e_1}\alpha_1 t}\\
=&\ \Phi_{e_1}^{\beta_0}(\xi,x)\left[\frac{1}{2}\zeta e^{\frac{1}{2}c_{e_1}\alpha_1t}-4c_{e_1}A_1^{1-{\beta_0}}e^{\alpha_1
(1-{\beta_0})r}e^{c_{e_1}\alpha_1(1-\beta_0)t}\right]\\
\geq &\ \Phi_{e_1}^{\beta_0}(\xi,x)\left[\frac{1}{2}\zeta e^{\frac{1}{2}c_{e_1}\alpha_1t}-4c_{e_1}A_1^{1-{\beta_0}}e^
{\alpha_1(1-{\beta_0})r}e^{\frac{1}{2}c_{e_1}\alpha_1t}\right]\\
=&\ \Phi_{e_1}^{\beta_0}(\xi,x)e^{\frac{1}{2}c_{e_1}\alpha_1t}\left[\frac{1}{2}\zeta -4c_{e_1}A_1^{1-{\beta_0}}e^{\alpha_1
(1-{\beta_0})r}\right].
\end{align*}
Therefore, there holds $\nabla_x u^+(x,t) \cdot n(x)\geq 0$. By a similar argument, one obtains that $\nabla_x u^-(x,t) 
\cdot n(x)\leq 0$. The proof is complete.
\end{proof}
	
\vspace{0.3cm}
\begin{proof}[\ \textbf{Proof of Theorem \bf{\ref{th2}}}] Take \(n\geq |T|+1\). Let $u_n(x,t)$ be the solution of problem 
\eqref{f1} with the initial value $u_n(x,-n)=u^+(x,-n)$, where $u^+(x,t)$ is given by Lemma \ref{lem5}. When $n$ is 
sufficiently large, we can conclude that
$$0\leq u_n(x,-n)=u^+(x,-n)\leq 1+\zeta\Vert\delta(x)\Vert_{L^\infty(\mathbb{R}^N)}e^{-\frac{1}{2}c_{e_1}\alpha_1 n}\leq 
1+\frac{1}{n}\ \ \text{for}\ x\in \overline{\Omega}.$$
Moreover,  by (H3)-(H4), it follows that $f(x,0)=0$ and $f(x,1+\frac{1}{n})<0$ for $x\in \mathbb{R}^N$. Therefore $0$ and 
$1+\frac{1}{n}$ can serve as a sub-solution and a super-solution of problem \eqref{f1}, respectively. By applying the 
comparison principle, we obtain that
$$0\leq u_n(x,t)\leq 1+\frac{1}{n}\ \ \text{for}\ x\in \overline{\Omega} \ \ \text{and}\ t> -n$$
when $n$ is sufficiently large.	Since $\Phi_{e_1}(\xi,x) $ is strictly decreasing with respect to $\xi$, it is easy to 
verify that
$$u^-(x,-n)\leq u_n(x,-n) \leq u^+(x,-n) \ \ \text{for all}\ x \in \overline{\Omega}.$$	
Then the comparison principle implies that
\begin{equation*}
u^-(x,t)\leq u_n(x,t) \leq u^+(x,t) \ \ \text{for all}\  x \in \overline{\Omega} \ \ \text{and}\  -n< t \leq T.
\end{equation*}
Setting $t = -n+1$ in the above inequality yields
$$u_n(-n+1, x) \leq u^+(-n+1, x) = u_{n-1}(-n+1, x).$$
Applying the comparison principle again, we obtain the following:
$$u_n(t, x) \leq u_{n-1}(t, x)\ \  \text{for all}\  x \in \overline{\Omega} \ \ \text{and}\  -n+1< t \leq T.$$
Furthermore, by applying standard parabolic estimates one obtains that as $n \rightarrow +\infty$, the solution sequence 
$u_n(x,t)$ converges to an entire solution $u(x,t)$ of problem \eqref{f1} locally uniformly for $(x,t)\in \overline
{\Omega}\times \mathbb{R}$. Then $u(x,t)$ satisfies
$$0\leq u(x,t) \leq 1 \ \ \text{for all}\  x \in \overline{\Omega} \ \ \text{and}\  t \in \mathbb{R} $$
and
\begin{equation}\label{qqq}
u^-(x,t)\leq u(x,t) \leq u^+(x,t) \ \text{ for all } x \in \overline{\Omega} \ \text{ and } t \leq T.
\end{equation}
By \eqref{qqq} one gets that
\begin{equation*}
\begin{aligned}
& \frac{\Phi_{e_1}(\xi^-,x)-\Phi_{e_1}(\xi,x)}{\Phi_{e_1}^{\beta_0}(\xi,x)}
-\zeta\delta(x)e^{\frac{1}{2}c_{e_1}\alpha_1t}\\
\leq & \frac{u(x,t)-\Phi_{e_1}(\xi,x)}{\Phi_{e_1}^{\beta_0}(\xi,x)}
\leq 
\frac{\Phi_{e_1}(\xi^+,x)-\Phi_{e_1}(\xi,x)}{\Phi_{e_1}^{\beta_0}(\xi,x)}
+\zeta\delta(x)e^{\frac{1}{2}c_{e_1}
\alpha_1t}
\end{aligned}
\end{equation*}
for all $x \in \overline{\Omega}$ and $t \leq T$. It then follows from the mean value theorem that
\begin{equation*}
\begin{split}
\frac{\Phi_{e_1}(\xi^+,x)-\Phi_{e_1}(\xi,x)}{\Phi_{e_1}^{\beta_0}(\xi,x)}\leq &\ M_1c_{e_1}\rho e^{\frac{1}{2}c_{e_1}
\alpha_1t} \frac{\Phi_{e_1}(\xi^++\tau(\xi-\xi^+),x)}{\Phi_{e_1}^{\beta_0}(\xi,x)}\\
\leq &\ M_1c_{e_1}\rho e^{\frac{1}{2}c_{e_1}\alpha_1t}\frac{\Phi_{e_1}(\xi^+,x)}{\Phi_{e_1}(\xi,x)},
\end{split}
\end{equation*}
where $0<\tau<1$ is a constant and $M_1$ is given by \eqref{M1M2}. Note that for all $x \in \mathbb{R}^N $ and  $t \in  
\mathbb{R}$, $\Phi_{e_1}(\xi-y,x)e^{-M_1y}$ is decreasing with respect to $y$. Hence
$$\frac{\Phi_{e_1}(\xi^+,x)}{\Phi_{e_1}(\xi,x)}\leq e^{M_1c_{e_1}\rho e^{\frac{1}{2}c_{e_1}\alpha_1t}} \ \ \text{for all}
\ x \in \overline{\Omega} \ \ \text{and}\ t \leq T,$$	
whence
\begin{equation}\label{b11}
\frac{\Phi_{e_1}(\xi^+,x)-\Phi_{e_1}(\xi,x)}{\Phi_{e_1}^{\beta_0}(\xi,x)}\leq M_1c_{e_1}\rho e^{\frac{1}{2}c_{e_1}\alpha_
1t}e^{M_1c_{e_1}\rho e^{\frac{1}{2}c_{e_1}\alpha_1t}} \ \ \text{for all}\ x \in\overline{\Omega} \ \ \text{and}\ t \leq T.
\end{equation}
Similarly, one has
\begin{equation}\label{b22}
\frac{\Phi_{e_1}(\xi^-,x)-\Phi_{e_1}(\xi,x)}{\Phi_{e_1}^{\beta_0}(\xi,x)}\geq -M_1c_{e_1}\rho e^{\frac{1}{2}c_{e_1}\alpha
_1t}\ \ \text{for all}\	x \in \overline{\Omega} \ \ \text{and}\ t \leq T.
\end{equation}
Letting $t \rightarrow -\infty$ in \eqref{b11} and \eqref{b22}, we derive that 
$$\frac{u(x,t)-\Phi_{e_1}(\xi,x)}{\Phi_{e_1}^{\beta_0}(\xi,x)}\rightarrow 0 \ \ \text{uniformly for}\ x\in \overline
{\Omega},$$
since $0<\Phi_{e_1}(\xi,x)<1$, \eqref{uxt1} holds.
		
Finally, $u_t(x,t)>0$ can be proved as that of  \cite[Theorem 3.2]{F1}. Since $ 0\leq u(x,t)\leq 1$ and $u_t(x,t)> 0$ for 
all $x\in \overline{\Omega}$ and $t\in \mathbb{R}$, it follows that $0< u(x,t)<1$ for all $x\in \overline{\Omega}$ and 
$t\in \mathbb{R}$. The proof is complete.
\end{proof}
	
\section{Long-time behavior of entire solutions}\label{sec4}
	
In this section, we first prove that the entire solution given by Theorem \ref{th2} propagates completely. Then, by 
constructing the super- and sub-solutions of problem \eqref{f1} for $t\geq0$, we prove that the entire solution is a 
transition front connecting 0 and 1, with a global mean speed $c_{e}$ identical to that of the pulsating front. 
Furthermore, we show that it is trapped between two translates of the pulsating front as $t\rightarrow +\infty$. Finally, 
under an appropriate assumption, we prove that the entire solution recovers its profile as $t\rightarrow+\infty$. 	

By the assumptions {(H3)}-{(H4)}, there is a function $g \in C^1([0,1],\mathbb{R})$ satisfying  
\begin{equation}\label{g}
\begin{cases}
g(0)=g(\theta_g)=g(1)=0,\ g'(0)<0,\ g'(1)<0,\ \int_0^1g(s)ds>0,\\
g(s)<0\ \ \text{for}\ 0<s<\theta_g,\ g(s)>0\ \ \text{for}\ \theta_g<s<1,\\
g(s)\leq f(x,s)\ \ \text{for}\ (x,s)\in \mathbb{R}^N\times [0,1],
\end{cases}
\end{equation}
where $\theta<\theta_g<1$ is a constant. It can then be inferred from Fife and McLeod \cite{P1}  that the following 
equation
\begin{equation}\label{g1}
u_t=\Delta u+g(u),\ \ x\in \mathbb{R}^N,\ t\in \mathbb{R}
\end{equation}
admits a unique (up to shift) planar front $U(x\cdot e-c_gt)$ with the speed $c_g>0$, where $c_g$ is independent of the 
direction $e\in \mathbb{S}^{N-1}$. Here we may also take $e=e_1$. Then $U(x\cdot e-c_gt)=U(x_1-c_gt)$.

For any $0<\epsilon<1-\theta_g$, $x_0\in\mathbb{R}^N$ and $R>0$, let us consider the following problem
\begin{equation}\label{g2}
\begin{cases}
u_t=\Delta u+g(u),\ \ x\in\Omega,\ t>0,\\
\frac{\partial u}{\partial n}=0,\ \ x\in\partial\Omega,\ t>0
\end{cases}
\end{equation}
with the initial value 
\begin{equation}\label{g3}
u(x,0)=
\begin{cases}
1-\epsilon,\ \ x\in\Omega\cap B(x_0,R),\\
0,\ \ x\in\Omega\setminus B(x_0,R),
\end{cases}
\end{equation}
where $g(u)$ and $\theta_g$ are given by \eqref{g}. It follows from the work of Berestycki, Hamel and Matano  \cite[Lemma 
5.2]{H1} that the following  conclusion holds.
\begin{lemma}\label{lem4}
Let $0<\epsilon<1-\theta_g$, $u_{\epsilon,x_0,R}(x,t)$ be the solution of equation \eqref{g2}-\eqref{g3}. Then there 
exist real numbers $0<R_1<R_2<R_3$ and $T_g>0$ such that $R_2-R_1>\frac{c_gT_g}{4}$, and if $B(x_0,R_3) \subset \Omega$, 
then 
$$u_{\epsilon,x_0,R_1}(x,T_g)\geq 1-\epsilon \ \ \text{for all}\  x \in \overline{B(x_0,R_2)},$$ 
where $c_g$ is the speed of the  planar front corresponding to equation \eqref{g1}.
\end{lemma}	
		
\begin{proof}[\ \textbf{Proof of Theorem \bf{\ref{th3}}}]
From Theorem \ref{th2}, one obtains that $0< u(x,t) <1 $ and $u_t(x,t)>0$ for all $x\in\overline{\Omega}$ and $t\in\mathbb
{R}$. Then, by standard parabolic estimates we can conclude that $u(x,t)$ converges to a solution $u_\infty(x)$ of the 
following problem
\begin{equation}\label{uinfty}
\begin{cases}
\Delta u+f(x,u)=0,\ \ x\in\Omega,\\
\frac{\partial u}{\partial n}=0,\ \ x\in\partial\Omega
\end{cases}
\end{equation}
locally uniformly for $x \in \overline{\Omega}$ as $t \rightarrow +\infty$ and $0<u_\infty(x)\leq 1$ for $x \in \overline
{\Omega}$.
		
{\it Step 1: proof of $\lim_{\vert x\vert \rightarrow +\infty}u_\infty(x)=1$.} Taking $0<\epsilon<1-\theta_g$, it follows 
from Theorem \ref{th2} that there exists $\widehat{T}<0$ such that 
$$\vert u(x,\widehat{T})-\Phi_{e_1}(x_1-c_{e_1}\widehat{T},x)\vert \leq  \frac{\epsilon}{2} \ \ \text{for all}\  x \in 
\overline{\Omega}.$$
Since $\Phi_{e_1}(\xi,x)$ is strictly decreasing in $\xi$, one gets that there exists $q_\epsilon>0$ such that
$$ u(x,\widehat{T})\geq \Phi_{e_1}(x_1-c_{e_1}\widehat{T},x)-\frac{\epsilon}{2}\geq 1-\epsilon \ \ \text{for all}\ x\in 
\{x\in\overline{\Omega}:x_1-c_{e_1}\widehat{T}\leq -q_\epsilon\}.$$
Let $u_{\epsilon,x_0,R}(x,t)$ be the solution of \eqref{g2}-\eqref{g3}, by Lemma \ref{lem4} there are real numbers $0<R_1
<R_2<R_3$ and $T_g>0$ such that $R_2-R_1>\frac{c_gT_g}{4}$, and if $B(x_0,R_3) \subset \Omega$, then
$$u_{\epsilon,x_0,R_1}(x,T_g)\geq 1-\epsilon \ \ \text{for all}\  x \in \overline{B(x_0,R_2)},$$
where $c_g$ is the speed of the  planar front corresponding to \eqref{g1}. Note that $K\subset B(0,r)$. For any point $y 
\in\overline{\Omega\setminus B(0,r+R_3-R_2)}$, there exists $k_y \in \mathbb{N}^+$ and $\{x^1,x^2,\cdots,x^k\}$ satisfying
$$B(x^1,R_1)\subset \{x\in\overline{\Omega}:x_1-c_{e_1}\widehat{T}\leq -q_\epsilon\}, B(x^i,R_3)\subset \Omega \ \ \text
{for}\ 1\leq i\leq k_y,$$
$$y \in B(x^{k_y},R_2),B(x^{i+1},R_1)\subset B(x^{i},R_2)\ \ \text{for}\  1\leq i\leq k_y-1.$$
By the choice of $x^1$, one has
$$u(x,\widehat{T})\geq 1-\epsilon \ \ \text{for}\  x\in B(x^1,R_1). $$
Since $g(u)\leq f(x,u)\ \text{ for }(x,u)\in \mathbb{R}^N\times [0,1]$, using the comparison principle one obtains that
$$u(x,t)\geq u_{\epsilon,x^1,R_1}(x,t) \ \ \text{for}\ x \in\overline{\Omega} \ \ \text{and}\ t>\widehat{T}.$$
Therefore, by Lemma \ref{lem4} there holds
$$u(x,\widehat{T}+T_g)\geq u_{\epsilon,x^1,R_1}(x,\widehat{T}+T_g)\geq 1-\epsilon \ \ \text{for}\ x \in B(x^1,R_2).$$
Since $B(x^2,R_1)\subset B(x^1,R_2)$, we can conclude that
$$u(x,\widehat{T}+T_g)\geq 1-\epsilon \ \ \text{for}\ x \in\overline{B(x^2,R_1)}.$$
Similarly, by reduction, one can conclude that
$$u(x,\widehat{T}+k_yT_g)\geq 1-\epsilon \ \ \text{for}\ x \in\overline{B(x^{k_y},R_2)}.$$
In particular, there holds $u(y,\widehat{T}+k_yT_g)\geq 1-\epsilon$.Since $u_t(x,t)>0$ for all $x \in \overline{\Omega}$ 
and $t \in \mathbb{R}$, we have
$$u_\infty(x)\geq 1-\epsilon \ \ \text{for all}\   x \in \overline{\Omega\setminus B(0,r+R_3-R_2)}.$$
Finally, together with $0< u_\infty(x) \leq 1 $ for all $x \in \overline{\Omega}$ and the arbitrariness of $\epsilon$, 
one obtains $\lim_{\vert x\vert \rightarrow +\infty}u_\infty(x)=1$.
		
{\it Step 2: proof of $u_\infty(x) \equiv 1$.}
Assume there is a point $x^* \in \overline{\Omega}$ such that $0<u_\infty(x^*)=\min_{x\in \overline{\Omega}}u_\infty(x)< 
1$. Define $v_\infty(x)=u_\infty(x)-u_\infty(x^*)$ for $x\in \overline{\Omega}$, then
$$v_\infty(x)\geq 0 \ \ \text{for}\  x\in \overline{\Omega}\ \ \text{and}\  v_\infty(x^*)=0.$$
By \eqref{uinfty}, we can compute
$$-\Delta v_\infty=f(x,u_\infty)-f(x,u_\infty(x^*))+f(x,u_\infty(x^*))\geq -Mv_\infty +f(x,u_\infty(x^*)) $$
for $x\in \Omega$, where $M$ is given by \eqref{M}. Then we immediately obtain that
\begin{equation*}
\begin{cases}
-\Delta v_\infty+Mv_\infty\geq 0,\ \ x\in\Omega,\\
\frac{\partial v_\infty}{\partial n}=0,\ \ x\in\partial\Omega.
\end{cases}
\end{equation*}
If $x^*\in\Omega$, then by the strong maximum principle, we have $v_\infty\equiv 0$ for $x\in\Omega$, which leads to a 
contradiction. If $x^*\in\partial\Omega$, then by Hopf's lemma one arrives at $\frac{\partial v_\infty}{\partial n}(x^*)
<0$, which contradicts $\frac{\partial v_\infty}{\partial n}=0$ for $x\in\partial\Omega$.

As a result, one has $u_\infty(x) \equiv 1$ for $x\in \overline{\Omega}$, which implies that
$$\lim_{t \rightarrow +\infty}u(x,t)=1 \ \ \text{locally uniformly for}\  x \in \overline{\Omega}.$$ 
This completes the proof.
\end{proof}
\vspace{0.3cm}
	
Let $\beta_1=\min\left\{\frac{1}{2},-\frac{\eta_1}{8M_1M_2}\right\}$. For any $\beta \in (0,\beta_1]$, by using the 
auxiliary function $\delta_0(x)$ given in Section \ref{sec3}, we define $\delta_\beta(x)=\delta_0(x)+C_\beta$, where 
$C_\beta$ is a constant chosen such that
\begin{equation}\label{deltabeta}
\begin{split}
\delta_\beta(x)&\geq 1, \\
\left\Vert\frac{\vert\nabla\delta_\beta(x)\vert}{\delta_\beta(x)}\right\Vert_{L^{\infty}(\mathbb{R}^N)}&\leq \min\left\{
\frac{c_{e_1}^2\beta}{64+100c_{e_1}},-\frac{\eta_1}{2(1+M_1+M_2)}\right\}, \\
\left\Vert\frac{\Delta\delta_\beta(x)}{\delta_\beta(x)}\right\Vert_{L^{\infty}(\mathbb{R}^N)}&\leq \min\left\{\frac{c_
{e_1}^2\beta}{64+100c_{e_1}},-\frac{\eta_1}{2(1+M_1+M_2)}\right\},
\end{split}
\end{equation}
where $\eta_1$ is given by \eqref{eta1theta1}, $M_1$ and $M_2$ are given by \eqref{M1M2}.
	
\begin{lemma}\label{lem6}
For any $\beta \in (0,\beta_1]$ and $0<\mu\leq \min\left\{\frac{c_{e_1}^2\beta}{64},-\frac{\eta_1}{4},c_{e_1}\alpha_2
\right\}$, there exists $\rho_0(\beta,\mu)>0 $, and for any $\rho\geq\rho_0(\beta,\mu)$ and $0<\zeta<\frac{\theta_1}
{2\Vert\delta_\beta(x)\Vert_{L^\infty(\mathbb{R}^N)}}$, there exists $\mathcal{T}_0(\beta,\rho,\zeta)>0$, such that for 
any $\mathcal{T}\geq \mathcal{T}_0(\beta,\rho,\zeta)$ the following functions
$$\overline{u}(x,t)=\Phi_{e_1}(\overline{\xi}(x,t),x)+\zeta\Phi_{e_1}^\beta(\overline{\xi}(x,t),x)\delta_\beta(x)e^{-\mu t}$$
and 
$$\underline{u}(x,t)=\Phi_{e_1}(\underline{\xi}(x,t),x)-\zeta\Phi_{e_1}^\beta(\underline{\xi}(x,t),x)\delta_\beta(x)e^{-\mu t}$$	
are the super-solution and sub-solution of problem \eqref{f1} for $x\in \overline{\Omega}$ and $t\geq 0$, respectively, 
where
$$\overline{\xi}(x,t)=x_1-c_{e_1}(t+\mathcal{T})-\rho \zeta(1-e^{-\mu t})$$ and
$$\underline{\xi}(x,t)=x_1-c_{e_1}(t+\mathcal{T})+\rho \zeta(1-e^{-\mu t})$$
with $\eta_1>0$, $\alpha_2>0$, $M_1>0$, $M_2>0$  given by \eqref{eta1theta1}, \eqref{A1},  \eqref{M1M2}, respectively.
\end{lemma}
	
\begin{proof}[\ \textbf{Proof}]
{\it Step 1: choice of parameters.}  Fix $\beta \in (0,\beta_1]$. Take
$$\rho_0(\beta,\mu)=\frac{(M-\eta_1)\Vert\delta_\beta(x)\Vert_{L^\infty(\mathbb{R}^N)}}{M_3\mu}\ \ \text{and}\ \  \mathcal
{T}_0(\beta,\rho,\zeta)=\max\bigg\{\frac{r+\rho\zeta}{c_{e_1}},T',T''\bigg\},$$
where $\eta_1>0$, $M>0$ and $M_3>0$ are given by \eqref{eta1theta1}, \eqref{M}, \eqref{C3M3}, respectively. The following 
details the selection of $\mathcal{T}_0(\beta,\rho,\zeta)$. It follows from $\mathcal{T}\geq \frac{r+\rho\zeta}{c_{e_1}}$ 
that $r-c_{e_1}\mathcal{T}\leq 0$. Since $K \subset B(0,r)$, it holds $\overline{\xi}(x,t)=x_1-c_{e_1}(t+\mathcal{T})-
\rho \zeta(1-e^{-\mu t})\leq 0$ for all $x \in \partial\Omega $ and $t\geq 0$. From \eqref{A1} one has
$$\vert\partial_\xi\Phi_{e_1}(\overline{\xi},x)\vert\leq A_1e^{\alpha_2\overline{\xi}},\vert\nabla_x\Phi_{e_1}(\overline
{\xi},x)\vert\leq A_1e^{\alpha_2\overline{\xi}} \ \ \text{for all}\  x \in \partial\Omega\ \ \text{and}\  t\geq 0.$$
Additionally, by \eqref{ddd} one gets that
$$\frac{\partial_\xi\Phi_{e_1}(\xi,x)}{\Phi_{e_1}(\xi,x)}\rightarrow 0 \ \ \text{as}\  t\rightarrow +\infty\ \ \text
{uniformly for}\   x \in \partial\Omega \subset B(0,r) $$
and
$$\frac{\vert\nabla_x\Phi_{e_1}(\xi,x)\vert}{\Phi_{e_1}(\xi,x)}\rightarrow 0\ \ \text{as}\ t\rightarrow +\infty\ \ \text
{uniformly for}\  x \in \partial\Omega \subset B(0,r). $$
Hence there is a sufficiently large positive constant $T'$ such that when $\mathcal{T}\geq T'$, there holds
$$\frac{\vert\partial_\xi\Phi_{e_1}(\overline{\xi},x)\vert}{\Phi_{e_1}(\overline{\xi},x)},\ \frac{\vert\nabla_x\Phi_{e_1}
(\overline{\xi},x)\vert}{\Phi_{e_1}(\overline{\xi},x)} \leq \frac{1}{4\beta \Vert\delta_\beta(x)\Vert_{L^\infty(\mathbb
{R}^N)}},$$
for all $x \in \partial\Omega$ and $t\geq 0$. Even if it means increasing $T'$, one has that when $\mathcal{T}\geq T'$,
$$2A_1e^{\alpha_2r-c_{e_1}\mathcal{T}}\leq \frac{1}{2}\zeta\Phi_{e_1}^\beta(\overline{\xi},x)\ \ \text{for all}\ x \in 
\partial\Omega\ \ \text{and}\ t\geq 0.$$
By a similar argument, there exists a sufficiently large $T''>0$ such that the same estimates hold when replacing 
$\overline{\xi}(x,t)$ by $\underline{\xi}(x,t)$. 
	
{\it Step 2: verifying  of the differential inequalities.} Let us prove $\mathcal{L} [\overline{u}](x,t)\geq 0$ for $x\in 
\Omega$ and $t\geq0$. $\mathcal{L} [\underline{u}](x,t)\leq 0$ can be handled in a similar manner. After some direct 
computations, one obtains
\begin{align*}
\mathcal{L}[\overline{u}](x,t)=&\ f(x,\Phi_{e_1}(\overline{\xi},x))-f(x,\overline{u}(x,t))-\rho\zeta\mu\partial_\xi\Phi_{e_ 
1}(\overline{\xi},x)e^{-\mu t}\\
&\ -\rho\zeta^2\beta\mu\Phi_{e_1}^{\beta-1}(\overline{\xi},x)\partial_\xi\Phi_{e_1}(\overline{\xi},x)\delta_\beta(x)e^{-2 
\mu t}-\zeta\mu\Phi_{e_1}^\beta (\overline{\xi},x)\delta_\beta(x)e^{-\mu t}\\
&\ -\zeta\Phi_{e_1}^\beta(\overline{\xi},x)\Delta \delta_\beta(x)e^{-\mu t}+\zeta\beta\Phi_{e_1}^{\beta-1}(\overline{\xi},x)
\delta_\beta (x)e^{-\mu t}f(x,\Phi_{e_1}(\overline{\xi},x))\\
&\ +\zeta\beta(1-\beta)\Phi_{e_1}^{\beta-2}(\overline{\xi},x)\delta_\beta(x)e^{-\mu t}\Big[\vert\partial_\xi\Phi_{e_1}
(\overline{\xi},x)\vert^2+\vert\nabla_x\Phi_{e_1}(\overline{\xi},x)\vert^2 \\
&\ +2\partial_\xi\Phi_{e_1}(\overline{\xi},x)\partial_{x_1}\Phi_{e_1}(\overline{\xi},x)\Big]-2\zeta\beta\Phi_{e_1}^{\beta- 
1}(\overline{\xi},x)\Big[ \partial_\xi\Phi_{e_1}(\overline{\xi},x)\partial_{x_1}\delta_\beta(x)\\
&\ +\nabla_x\Phi_{e_1}(\overline{\xi},x)\cdot \nabla\delta_\beta(x)\Big]e^{-\mu t}\\
\geq&\ f(x,\Phi_{e_1}(\overline{\xi},x))-f(x,\overline{u}(x,t))-\rho\zeta\mu\partial_\xi\Phi_{e_1}(\overline{\xi},x)e^{- 
\mu t}\\
&\ +\zeta\Phi_{e_1}^\beta(\overline{\xi},x)\delta_\beta(x)e^{-\mu t}\bigg[\beta(1-\beta)\frac{\vert\partial_\xi\Phi_{e_1}
(\overline{\xi},x)\vert^2}{\Phi_{e_1}^2(\overline{\xi},x)}-\mu-2\beta(1-\beta) \frac{\vert\partial_\xi\Phi_{e_1}
(\overline{\xi},x)\vert}{\Phi_{e_1}(\overline{\xi},x)}\times\\
&\ \frac{\vert\nabla_x\Phi_{e_1}(\overline{\xi},x)\vert}{\Phi_{e_1}(\overline{\xi},x)}-\frac{\vert\Delta\delta_\beta(x) 
\vert}{\delta_\beta(x)}-2\beta\frac{\vert\nabla\delta_\beta(x)\vert}{\delta_\beta(x)} \frac{\vert\partial_\xi\Phi_{e_1}
(\overline{\xi},x)\vert}{\Phi_{e_1}(\overline{\xi},x)}-2\beta  \frac{\vert\nabla_x\Phi_{e_1}(\overline{\xi},x)\vert}{\Phi_
{e_1}(\overline{\xi},x)} \frac{\vert\nabla\delta_\beta(x)\vert}{\delta_\beta(x)} \bigg].
\end{align*}
In the sequel, we consider the three cases: $ \overline{\xi}(x,t)\geq C_3 $, $ \overline{\xi}(x,t)\leq -C_3 $ and $ -C_3
\leq\overline{\xi}(x,t)\leq C_3 $, where $C_3$ is given by \eqref{C3M3}.
		
{\bf Case 1: $\overline{\xi}(x,t)\geq C_3>C$.} By \eqref{CM0} and $0<\zeta<\frac{\theta_1}{2\Vert\delta_\beta(x)\Vert_
{L^\infty(\mathbb{R}^N)}}$, one derives that
$$0<\Phi_{e_1}(\overline{\xi},x)<\overline{u}(x,t)\leq \theta_1,$$
as a result, $f(x,\Phi_{e_1}(\overline{\xi},x))=f(x,\overline{u}(x,t))=0$, which yields
\begin{equation*}
\begin{split}
\mathcal{L}[\overline{u}](x,t)\geq&\ \zeta\Phi_{e_1}^\beta(\overline{\xi},x)\delta_\beta(x)e^{-\mu t}\bigg[ \beta(1-\beta)
\frac{\vert\partial_\xi\Phi_{e_1}(\overline{\xi},x)\vert^2}{\Phi_{e_1}^2(\overline{\xi},x)}-\mu-2\beta(1-\beta) \frac{
\vert\partial_\xi\Phi_{e_1}(\overline{\xi},x)\vert}{\Phi_{e_1}(\overline{\xi},x)}\times\\
&\ \frac{\vert\nabla_x\Phi_{e_1}(\overline{\xi},x)\vert}{\Phi_{e_1}(\overline{\xi},x)}-\frac{\vert\Delta\delta_\beta(x) 
\vert}{\delta_\beta(x)}-2\beta \frac{\vert\nabla\delta_\beta(x)\vert}{\delta_\beta(x)} \frac{\vert\partial_\xi\Phi_{e_1}
(\overline{\xi},x)\vert}{\Phi_{e_1}(\overline{\xi},x)}\\
&\ -2\beta \frac{\vert\nabla_x\Phi_{e_1}(\overline{\xi},x)\vert}{\Phi_{e_1}(\overline{\xi},x)} \frac{\vert\nabla\delta_
\beta(x)\vert}{\delta_\beta(x)} \bigg].
\end{split}
\end{equation*}
From \eqref{C3M3}, \eqref{deltabeta} and $\beta\leq \frac{1}{2}$, we can conclude that $\mathcal{L} [\overline{u}](x,t)
\geq 0$ due to $\mu\leq\frac{c_{e_1}^2\beta}{64}$.
		
{\bf Case 2: $\overline{\xi}(x,t)\leq -C_3<-C$.} By \eqref{CM0} one deduces that
$$1-\frac{1}{2}\theta_1\leq\Phi_{e_1}(\overline{\xi},x))<\overline{u}(x,t)\leq 1+\frac{1}{2}\theta_1.$$
Since $f_u(x,u)\leq \eta_1<0$ for $(x,u) \in \mathbb{R}^N\times[1-\theta_1,1+\theta_1]$, one has $ f(x,u_1)-f(x,u_2)\geq 
-\eta_1 (u_2-u_1)$ for all $x\in \mathbb{R}^N$ and $1-\theta_1\leq u_1\leq u_2\leq 1+\theta_1$, hence
\begin{equation*}
\begin{split}
\mathcal{L}[\overline{u}](x,t)\geq&\ \zeta\Phi_{e_1}^\beta(\overline{\xi},x)\delta_\beta(x)e^{-\mu t}\bigg[-\eta_1 -\mu-
2\beta(1-\beta) \frac{\vert\partial_\xi\Phi_{e_1}(\overline{\xi},x)\vert}{\Phi_{e_1}(\overline{\xi},x)} \frac{\vert\nabla_
x\Phi_{e_1}(\overline{\xi},x)\vert}{\Phi_{e_1}(\overline{\xi},x)}\\
&\ -\frac{\vert\Delta\delta_\beta(x)\vert}{\delta_\beta(x)}-2\beta \frac{\vert\nabla\delta_\beta(x)\vert}{\delta_\beta(x)} 
\frac{\vert\partial_\xi\Phi_{e_1}(\overline{\xi},x)\vert}{\Phi_{e_1}(\overline{\xi},x)}-2\beta  \frac{\vert\nabla_x\Phi_
{e_1}(\overline{\xi},x)\vert}{\Phi_{e_1}(\overline{\xi},x)} \frac{\vert\nabla\delta_\beta(x)\vert}{\delta_\beta(x)} 
\bigg].
\end{split}
\end{equation*}
It then follows from $\beta\leq \min\left\{\frac{1}{2},-\frac{\eta_1}{8M_1M_2}\right\}$, \eqref{M1M2}, \eqref{deltabeta} 
and $\mu\leq-\frac{\eta_1}{4}$ that $\mathcal{L} [\overline{u}](x,t)\geq 0$. 

{\bf Case 3: $-C_3\leq \overline{\xi}(x,t)\leq C_3 $.} It holds
$$f(x,\Phi_{e_1}(\overline{\xi},x))-f(x,\overline{u}(x,t))\geq-M\zeta\Phi_{e_1}^\beta(\overline{\xi},x)\delta_\beta(x)e^ 
{-\mu t},$$
where $M$ is given by \eqref{M}. Then, by \eqref{C3M3} one has
\begin{equation*}
\begin{split}
\mathcal{L}[\overline{u}](x,t)\geq&\ \zeta\Phi_{e_1}^\beta(\overline{\xi},x)\delta_\beta(x)e^{-\mu t}\bigg[\frac{M_3\rho 
\mu}{\Vert\delta_\beta(x)\Vert_{L^\infty(\mathbb{R}^N)}}-M -\mu-2\beta(1-\beta) \frac{\vert\partial_\xi\Phi_{e_1}
(\overline{\xi},x)\vert}{\Phi_{e_1}(\overline{\xi},x)}\times\\
&\ \frac{\vert\nabla_x\Phi_{e_1}(\overline{\xi},x)\vert}{\Phi_{e_1}(\overline{\xi},x)}-\frac{\vert\Delta\delta_\beta(x) 
\vert}{\delta_\beta(x)}-2\beta \frac{\vert\nabla\delta_\beta(x)\vert}{\delta_\beta(x)} \frac{\vert\partial_\xi\Phi_{e_1}
(\overline{\xi},x)\vert}{\Phi_{e_1}(\overline{\xi},x)}\\
&\ -2\beta \frac{\vert\nabla_x\Phi_{e_1}(\overline{\xi},x)\vert}{\Phi_{e_1}(\overline{\xi},x)} \frac{\vert\nabla\delta_
\beta(x)\vert}{\delta_\beta(x)} \bigg].
\end{split}
\end{equation*}
The choice of $\rho_0(\beta,\mu)$ implies that $\frac{M_3\rho \mu}{\Vert\delta_\beta(x)\Vert_{L^\infty(\mathbb{R}^N)}}
-M\geq -\eta_1$. Similar to Case 2, one arrives at $\mathcal{L} [\overline{u}](x,t)\geq0$.

{\it Step 3: verifying the boundary condition.} Some computations imply that
\begin{equation*}
\begin{split}
\nabla_x \overline{u}(x,t) \cdot n(x)=&\ \partial_\xi\Phi_{e_1}(\overline{\xi},x) n_1(x)+\nabla_x\Phi_{e_1}(\overline{\xi}
,x)\cdot n(x)\\
&\ +\zeta\beta\Phi_{e_1}^{\beta-1}(\overline{\xi},x)\delta_\beta(x)e^{-\mu t}[\partial_\xi\Phi_{e_1}(\overline{\xi},x) 
n_1(x)\\
&\ +\nabla_x\Phi_{e_1}(\overline{\xi},x)\cdot n(x)]+\zeta\Phi_{e_1}^\beta(\overline{\xi},x)e^{-\mu t}\nabla\delta(x)\cdot 
n(x).
\end{split}
\end{equation*}	
By the choice of $\mathcal{T}_0(\beta,\rho,\zeta)$, we can conclude that when $\mathcal{T}\geq\mathcal{T}_0(\beta,\rho,
\zeta)$,
\begin{equation*}
\begin{split}
\nabla_x \overline{u}(x,t) \cdot n(x)\geq&\ -2A_1e^{\alpha_2\overline{\xi}}-\zeta\beta\Vert\delta_\beta(x)\Vert_{L^\infty
(\mathbb{R}^N)}\frac{\vert\partial_\xi\Phi_{e_1}(\overline{\xi},x)\vert}{\Phi_{e_1}(\overline{\xi},x)}\Phi_{e_1}^\beta
(\overline{\xi},x)e^{-\mu t}\\
&\ -\zeta\beta\Vert\delta_\beta(x)\Vert_{L^\infty(\mathbb{R}^N)}\frac{\vert\nabla_x\Phi_{e_1}(\overline{\xi},x)\vert}{\Phi_ 
{e_1}(\overline{\xi},x)}\Phi_{e_1}^\beta(\overline{\xi},x)e^{-\mu t}+\zeta\Phi_{e_1}^\beta(\overline{\xi},x)e^{-\mu t}\\
\geq&\  \zeta\Phi_{e_1}^\beta(\overline{\xi},x)e^{-\mu t}-2A_1e^{\alpha_2r-c_{e_1}\mathcal{T}}e^{-c_{e_1}\alpha_2t}-\frac
{1}{2}\zeta\Phi_{e_1}^\beta(\overline{\xi},x)e^{-\mu t}\\
\geq& \ \frac{1}{2}\zeta\Phi_{e_1}^\beta(\overline{\xi},x)\left[e^{-\mu t}-e^{-c_{e_1}\alpha_2t}\right]
\end{split}
\end{equation*}
for all $x \in \partial\Omega$ and $t\geq 0$. Since $\mu\leq c_{e_1}\alpha_2$, we can get $\nabla_x \overline{u}(x,t) 
\cdot n(x)\geq 0$ for all $x \in \partial\Omega$ and $t\geq 0$. Similarly, we can get $\nabla_x \underline{u}(x,t) \cdot 
n(x)\leq 0$. The proof is complete.
\end{proof}
	
\begin{lemma}\label{lem7}
Let $u(x,t)$ be given by Theorem $\ref{th2}$. For any $0<\epsilon<1-\theta_g$, there exist $T_\epsilon>0$, $Q_\epsilon^1
>0$ and $Q_\epsilon^2>0$ such that for any $t\geq T_\epsilon$,
$$u(x,t)\leq \epsilon \ \ \text{for all}\   x\in\{x\in\overline{\Omega}:x_1-c_{e_1}t\geq Q_\epsilon^1\} $$
and
$$u(x,t)\geq 1-\epsilon \ \ \text{for all}\   x\in\{x\in\overline{\Omega}:x_1-c_{e_1}t\leq -Q_\epsilon^2\},$$
where $\theta_g$ is given by \eqref{g}.
\end{lemma}
	
\begin{proof}[\ \textbf{Proof}]
Take $\beta\in(0,\min\{\beta_0,\beta_1\}]$ and define $\delta_\beta(x)$ satisfying \eqref{deltabeta}. For any fixed $0<
\epsilon<1-\theta_g$, by Lemma \ref{lem6}, for any $0<\mu\leq\min\left\{\frac{c_{e_1}^2\beta}{64},-\frac{\eta_1}{4},
c_{e_1}\alpha_2\right\}$, there exists $\rho_0(\beta,\mu)>0 $, and for any $\rho\geq\rho_0(\beta,\mu)$ and $0<\zeta<\min
\left\{\frac{\epsilon}{2\Vert\delta_\beta(x)\Vert_{L^\infty(\mathbb{R}^N)}},\frac{\theta_1}{2\Vert\delta_\beta(x)\Vert_ 
{L^\infty(\mathbb{R}^N)}}\right\}$, there exists $\mathcal{T}_0(\beta,\rho,\zeta)>0$, such that for any $\mathcal{T}_0
\geq \mathcal{T}(\beta,\rho,\zeta)$ the following functions
$$\overline{u}(x,t)=\Phi_{e_1}(\overline{\xi}(x,t),x)+\zeta\Phi_{e_1}^{\beta}(\overline{\xi}(x,t),x)\delta_\beta(x)e^{- 
\mu t}$$
and
$$\underline{u}(x,t)=\Phi_{e_1}(\underline{\xi}(x,t),x)-\zeta\Phi_{e_1}^{\beta}(\underline{\xi}(x,t),x)\delta_\beta(x)e^ 
{-\mu t}$$
are the super-solution and sub-solution of problem \eqref{f1} for all $x\in\overline{\Omega}$ and $t\geq0$, respectively, 
where
$$\overline{\xi}(x,t)=x_1-c_{e_1}(t+\mathcal{T})-\rho \zeta(1-e^{-\mu t}),$$
$$\underline{\xi}(x,t)=x_1-c_{e_1}(t+\mathcal{T})+\rho \zeta(1-e^{-\mu t}),$$
$\eta_1$, $\alpha_2$ and $M_1,M_2>0$ are given by \eqref{eta1theta1}, \eqref{M1M2} and \eqref{A1}. It follows from  
Theorem \ref{th2} that there exists $\mathcal{T}'<0$ such that
$$u(x,\mathcal{T}')\leq \Phi_{e_1}(x_1-c_{e_1}\mathcal{T}',x)+\zeta\Phi_{e_1}^{\beta}(x_1-c_{e_1}\mathcal{T}',x)\ \ 
\text{for all}\ x\in\overline{\Omega}.$$
Since $\Phi_{e_1}(\xi,x)$ is strictly decreasing in $\xi$, there holds
\begin{equation*}
\begin{split}
u(x,\mathcal{T}')&\ \leq \Phi_{e_1}(x_1-c_{e_1}\mathcal{T},x)+\zeta\Phi_{e_1}^{\beta}(x_1-c_{e_1}\mathcal{T},x)\\
&\ \leq\Phi_{e_1}(x_1-c_{e_1}\mathcal{T},x)+\zeta\Phi_{e_1}^{\beta}(x_1-c_{e_1}\mathcal{T},x){\delta_\beta}(x)\\
&\ =\overline{u}(x,0)
\end{split}
\end{equation*}
for all $x\in\overline{\Omega}$. It then follows from the comparison principle that
\begin{equation}\label{fdf}
u(x,t)\leq \overline{u}(x,t-\mathcal{T}')\ \ \text{for all}\ x\in\overline{\Omega}\ \ \text{and}\ t\geq \mathcal{T}'.
\end{equation}
Since $\lim_{\xi \to +\infty}\Phi_{e_1}(\xi,x)=0$ uniformly for $x\in\mathbb{R}^N$, by \eqref{fdf} we can derive that 
there exists $Q_\epsilon^1>0$ sufficiently large such that for any $t> \mathcal{T}'$,
\begin{equation}\label{q1}
u(x,t)\leq \frac{\epsilon}{2}+\zeta\delta_\beta(x)\leq \epsilon \ \ \text{for all}\ x\in\{x\in\overline{\Omega}:x_1-
c_{e_1}t\geq Q_\epsilon^1\}.
\end{equation}
It follows from Theorem \ref{th2} that there exists $\widehat{T}<0$ such that
$$u(x,\widehat{T})\geq \Phi_{e_1}(x_1-c_{e_1}\widehat{T},x)-\frac{\zeta}{2}.$$
Since $\Phi_{e_1}(\xi,x)$ is strictly decreasing in $\xi$, there exists a real number $q_1>0$ such that
$$ \Phi_{e_1}(x_1 -c_{e_1}\widehat{T},x)\geq 1-\frac{\zeta}{2}\ \ \text{for all}\ x\in\{x\in\overline{\Omega}:x_1\leq 
-q_1\}.$$
Therefore
$$u(x,\widehat{T})\geq 1-{\zeta}\ \ \text{for all}\ x\in\{x\in\overline{\Omega}:x_1\leq -q_1\}.$$
By using Lemma \ref{lem4}, similar to the proof of Theorem \ref{th3}, we can prove that there exists $k\in\mathbb{N}^+$ 
such that
$$u(x,\widehat{T}+kT_g)\geq 1-\zeta \ \ \text{for all}\ x\in\overline{\Omega\setminus B(0,r+R_3-R_2)}\cap\{x\in\overline
{\Omega}:x_1\leq q_2\},$$
where $q_2>0$ is a sufficiently large constant such that $\underline{u}(x,0)\leq 0$ for all $x\in\{x\in\overline{\Omega}:
x_1\geq q_2\}$.
		
Additionally,  by Theorem \ref{th3}, one gets that there exists $\widehat{T}'>0$ such that
$$u(x,\widehat{T}')\geq 1-\zeta \ \ \text{for all}\ x\in\overline{ B(0,r+R_3-R_2)\setminus K}.$$
From the above estimates, we have
$$u(x,t)\geq 1-\zeta \ \ \text{for all}\  x\in \{x\in\overline{\Omega}:x_1\leq q_2\}\  \ \text{and}\ t\geq \max\{\widehat
{T}+kT_g,\widehat{T}'\}.$$
Therefore, by taking $\mathcal{T}''>\max\{\widehat{T}+kT_g,\widehat{T}',\mathcal{T}\}$, we have
\begin{equation*}
\begin{split}
u(x,\mathcal{T}'')\geq&\ 1-\zeta\\ 
\geq&\ \Phi_{e_1}(x_1-c_{e_1}\mathcal{T},x)-\zeta\Phi_{e_1}(x_1-c_{e_1}\mathcal{T},x)\\
\geq&\ \Phi_{e_1}(x_1-c_{e_1}\mathcal{T},x)-\zeta\Phi_{e_1}^{\beta}(x_1-c_{e_1}\mathcal{T},x)\delta_\beta(x)\\
=&\ \underline{u}(x,0)
\end{split}
\end{equation*}
for all $x\in \{x\in\overline{\Omega}:x_1\leq q_2\}$. 
In addition, for $x\in 
\{x\in\overline{\Omega}:x_1\geq q_2\}$, we 
have
$$u(x,\mathcal{T}'')\geq 0\geq\underline{u}(x,0),$$
hence
$$u(x,\mathcal{T}'')\geq\underline{u}(x,0)\ \ \text{for all}\ x\in\overline{\Omega}.$$
It follows from the comparison principle  that
\begin{equation}\label{lll}
u(x,t)\geq \underline{u}(x,t-\mathcal{T}'')\ \ \text{for all}\ x\in\overline{\Omega}\ \ \text{and}\ t\geq \mathcal{T}''.
\end{equation}
Since $\lim_{\xi \to -\infty}\Phi_{e_1}(\xi,x)=1$ uniformly for $x\in\mathbb{R}^N$, by \eqref{lll} we can derive that 
there exists $Q_\epsilon^2>0$ sufficiently large such that for any $t\geq \mathcal{T}''$,
\begin{equation}\label{q2}
u(x,t)\geq 1-\frac{\epsilon}{2}-\zeta\delta_\beta(x)\geq 1-\epsilon \ \ \text{for all}\ x\in\{x\in\overline{\Omega}:x_1-
c_{e_1}t\leq -Q_\epsilon^2\}.
\end{equation}
Combining \eqref{q1} and \eqref{q2}, by taking $T_\epsilon\geq\mathcal{T}''$, we get the desired result. The proof is 
complete.
\end{proof}
	
\vspace{0.3cm}
\begin{proof}[\ \textbf{Proof of Theorem \bf{\ref{th4}}}]
Based on Lemma \ref{lem7}, Theorem \ref{th4} can be proved similarly to that of \cite[Theorem 1.5]{F1}. The proof is 
complete.
\end{proof}
	
It follows from  \cite[Lemma 2.4]{F1} that the following lemma holds.
\begin{lemma}
Let $U(x,t)$ be a solution of equation \eqref{f2} satisfying
$$\Phi_{e_1}(x_1 - c_{e_1} t + \tau_1, x) \leq U( x,t) \leq \Phi_{e_1}(x_1 - c_{e_1} t + \tau_2, x)$$
for all $(x,t)\in\mathbb{R}^N\times\mathbb{R} $, where $\tau_1$ and $\tau_2$ are two constants satisfying $\tau_2\leq  0
\leq\tau_1$. Assume further that for any $\tau>0$, there exists $Q_\tau>0$ such that
\begin{equation*}
\Phi_{e_1}(x_1-c_{e_1}t+\tau,x)\leq U(x,t)\leq \Phi_{e_1}(x_1-c_{e_1}t-\tau,x)\ \ \text{for}\ (x,t)\in\mathbb{R}^N\times
\mathbb{R}\ \ \text{with}\  x_1-c_{e_1}t\geq Q_\tau.
\end{equation*}
Then $U(x,t)\equiv \Phi_{e_1}(x_1-c_{e_1}t,x)$ for all $(x,t)\in\mathbb{R}^N\times\mathbb{R} $.
\end{lemma}
	
\begin{proof}[\ \textbf{Proof of Theorem \bf{\ref{th5}}}]
Take $\beta\in(0,\min\{\beta_0,\beta_1\}]$. Assume that there are some $\epsilon_0>0$ and a sequence $\{(x_n,t_n)\}\in 
\overline{\Omega}\times \mathbb{R}$ such that $t_n\rightarrow +\infty$ as $n\rightarrow +\infty$ and
\begin{equation}\label{ooo}
\left\vert\frac{u(x_n,t_n)-\Phi_{e_1}(x_{n1}-c_{e_1}t_n,x_n)}{\Phi_{e_1}^\beta(x_{n1}-c_{e_1}t_n,x_n)}\right\vert\geq 
\epsilon_0 \ \ \text{or all}\ n\in\mathbb{N},
\end{equation}
where $x_{n1}$ denotes the first component of $x_n$. If $\vert x_n \vert < +\infty$ as $n\rightarrow +\infty$, then by 
Theorem \ref{th4} and $\lim_{\xi \to -\infty}\Phi_{e_1}(\xi,x)=1$ (uniformly for $x\in\mathbb{R}^N$) we have
$$u(x_n,t_n),\Phi_{e_1}(x_{n1}-c_{e_1}t_n,x_n)\rightarrow 1\ \ \text{as}\ n\rightarrow +\infty,$$
which contradicts \eqref{ooo}. Therefore, it holds $\vert x_n \vert \rightarrow +\infty$ as $n\rightarrow +\infty$.	It 
can be discussed in a similar manner to conclude that the sequence $\xi_n:=x_{n1}-c_{e_1}t_n$ cannot diverge to 
$-\infty$. Consider the case $\xi_n  \rightarrow +\infty$ as $n\rightarrow +\infty$. By \eqref{fdf} and \eqref{lll} we 
have
\begin{align*}
&\ \frac{\Phi_{e_1}(x_{n1}-c_{e_1}(t_n+\mathcal{T}-\mathcal{T}'')+\rho\zeta (1-e^{-\mu(t_n-\mathcal{T}'')}),x_n)-\Phi_{e_1}
(x_{n1}-c_{e_1}t_n,x_n)}{\Phi_{e_1}^\beta(x_{n1}-c_{e_1}t_n,x_n)}\\
&\ -\frac{\zeta\Phi_{e_1}^\beta(x_{n1}-c_{e_1}(t_n+\mathcal{T}-\mathcal{T}'')+\rho\zeta (1-e^{-\mu(t_n-\mathcal{T}'')}),x_n)
\delta_\beta(x_n)e^{-\mu(t_n-\mathcal{T}'')}}{\Phi_{e_1}^\beta(x_{n1}-c_{e_1}t_n,x_n)}\\
\leq&\ \frac{u(x_n,t_n)-\Phi_{e_1}(x_{n1}-c_{e_1}t_n,x_n)}{\Phi_{e_1}^\beta(x_{n1}-c_{e_1}t_n,x_n)}\\
\leq&\ \frac{\Phi_{e_1}(x_{n1}-c_{e_1}(t_n+\mathcal{T}-\mathcal{T}')-\rho\zeta (1-e^{-\mu(t_n-\mathcal{T}')}),x_n)-\Phi_
{e_1}(x_{n1}-c_{e_1}t_n,x_n)}{\Phi_{e_1}^\beta(x_{n1}-c_{e_1}t_n,x_n)}\\
&\ +\frac{\zeta\Phi_{e_1}^\beta(x_{n1}-c_{e_1}(t_n+\mathcal{T}-\mathcal{T}')-\rho\zeta (1-e^{-\mu(t_n-\mathcal{T}')}),x_n)
\delta_\beta(x_n)e^{-\mu(t_n-\mathcal{T}')}}{\Phi_{e_1}^\beta(x_{n1}-c_{e_1}t_n,x_n)}.
\end{align*}
Then, by $\partial_\xi \Phi_{e_1}(\xi,x)<0$ one concludes that
\begin{align*}
&\ -\frac{\Phi_{e_1}(x_{n1}-c_{e_1}t_n,x_n)}{\Phi_{e_1}^\beta(x_{n1}-c_{e_1}t_n,x_n)}-\frac{\zeta\Phi_{e_1}^\beta(x_{n1}-c_ 
{e_1}(t_n+\mathcal{T}-\mathcal{T}''),x_n)\delta_\beta(x_n)e^{-\mu(t_n-\mathcal{T}'')}}{\Phi_{e_1}^\beta(x_{n1}-c_{e_1}t_n 
, x_n)}\\
\leq&\ \frac{u(x_n,t_n)-\Phi_{e_1}(x_{n1}-c_{e_1}t_n,x_n)}{\Phi_{e_1}^\beta(x_{n1}-c_{e_1}t_n,x_n)}\\
\leq&\ \frac{\Phi_{e_1}(x_{n1}-c_{e_1}(t_n+\mathcal{T}-\mathcal{T}')-\rho\zeta,x_n) }{\Phi_{e_1}^\beta(x_{n1}-c_{e_1}t_n,
x_n)}\\
&\ +\frac{\zeta\Phi_{e_1}^\beta(x_{n1}-c_{e_1}(t_n+\mathcal{T}-\mathcal{T}')-\rho\zeta ,x_n)\delta_\beta(x_n)e^{-\mu(t_n-
\mathcal{T}')}}{\Phi_{e_1}^\beta(x_{n1}-c_{e_1}t_n,x_n)}.
\end{align*}
Since $0<\beta<1$, letting $n\rightarrow +\infty$, we can get that
$$\frac{u(x_n,t_n)-\Phi_{e_1}(x_{n1}-c_{e_1}t_n,x_n)}{\Phi_{e_1}^\beta(x_{n1}-c_{e_1}t_n,x_n)}\rightarrow 0,$$
which contradicts \eqref{ooo}. Therefore, we obtain that $\xi_n  \rightarrow \xi^*\in\mathbb{R}$ as $n\rightarrow 
+\infty$.
		
Decompose $x_n$ into $\overline{x_n}+\widehat{x_n}$, where $\overline{x_n}\in \{(k_1L_1,\cdots,k_NL_N):k_1,\cdots,k_N\in
\mathbb{Z}\}$ and $\widehat{x_n}\in[0,L_1)\times\cdots \times[0,L_N)$. Without loss of generality, we assume that 
$\widehat{x_n}\rightarrow x^*\in [0,L_1]\times\cdots\times [0,L_N]$ as $n\rightarrow +\infty$. Let $s_n:=t_n+\frac{\xi^*-
x^*_1}{c_{e_1}}$ for all $n\in\mathbb{N}$. Define
$$u_n(x,t)=u(x+\overline{x_n},t+s_n)\ \ \text{for all}\ t\in\mathbb{R} \ \ \text{and}\ x\in\Omega-\{\overline{x_n}\}.$$
Note that $f(x,u)$ is $L$-periodic with respect to $x$.	Since $K=\mathbb{R}^N\setminus\Omega$ is a compact set and $\vert 
 x_n\vert\rightarrow +\infty$ as $n\rightarrow +\infty$, standard parabolic estimates yield that as $n\rightarrow+
\infty$, the sequence $u_n(x,t)$ converges to an entire solution $U(x,t)$ of equation \eqref{f2} locally uniformly for 
$(x,t)\in\mathbb{R}^N\times\mathbb{R}$.
		
Additionally, note that $\Phi_{e_1}(\xi,x)$ is $L$-periodic with respect to $x$. Then, by \eqref{fdf} and \eqref{lll} in 
the proof of Lemma \ref{lem7} we can derive that
\begin{equation}\label{unxt}
\begin{split}
&\ \Phi_{e_1}(x_1-c_{e_1}t+\overline{x_n}_1-c_{e_1}s_n+\tau''-\rho\zeta e^{-\mu(t+s_n-\mathcal{T}'')},x)\\
&\ -\zeta\Phi_{e_1}^\beta(x_1-c_{e_1}t+\overline{x_n}_1-c_{e_1}s_n+\tau''-\rho\zeta e^{-\mu(t+s_n-\mathcal{T}'')},x)\delta_
\beta(x+\overline{x_n})e^{-\mu(t+s_n-\mathcal{T}'')}\\
\leq&\ u_n(x,t)\\
\leq&\ \Phi_{e_1}(x_1-c_{e_1}t+\overline{x_n}_1-c_{e_1}s_n+\tau'+\rho\zeta e^{-\mu(t+s_n-\mathcal{T}')},x)\\
&\ +\zeta\Phi_{e_1}^\beta(x_1-c_{e_1}t+\overline{x_n}_1-c_{e_1}s_n+\tau'+\rho\zeta e^{-\mu(t+s_n-\mathcal{T}')},x)\delta_
\beta(x+\overline{x_n})e^{-\mu(t+s_n-\mathcal{T}')}\\
\end{split}
\end{equation}
for all $t\geq \mathcal{T}''-s_n$ and $x\in \Omega-\{\overline{x_n}\}$, where $\tau'=c_{e_1}(\mathcal{T}'-\mathcal{T})-
\rho\zeta$, $\tau''=c_{e_1}(\mathcal{T}''-\mathcal{T})+\rho\zeta$. Letting $n\rightarrow +\infty$ in \eqref{unxt}, we have
$$\Phi_{e_1}(x_1-c_{e_1}t+\tau'',x)\leq U(x,t)\leq \Phi_{e_1}(x_1-c_{e_1}t+\tau',x)$$
for all $(x,t)\in \mathbb{R}^N\times\mathbb{R}$. Since $u(x,t)$ satisfies the assumption \eqref{jkl}, we claim that for 
any $\tau>0$, there is $Q_\tau>0$ such that for any $t\in \mathbb{R}$,
\begin{equation}\label{bnm}
\Phi_{e_1}(x_1-c_{e_1}t+\tau,x)\leq u(x,t)\leq \Phi_{e_1}(x_1-c_{e_1}t-\tau,x)\ \ \text{for all}\ x\in\{x\in \overline
{\Omega}:x_1-c_{e_1}t\geq Q_\tau\}.
\end{equation}
		
The proof of claim \eqref{bnm} will be given later. Assume that there are some $\tau_0>0$ and a sequence $\{(x_n,t_n)\}
\subset\overline{\Omega}\times\mathbb{R}$ satisfying $x_{n1}-c_{e_1}t_n\rightarrow +\infty$ as $n\rightarrow +\infty$ 
such that for all $n\in\mathbb{R}$,
$$u(x_n,t_n)>\Phi_{e_1}(x_{n1}-c_{e_1}t_n-\tau_0,x_n)\ \ \text{or}\ u(x_n,t_n)<\Phi_{e_1}(x_{n1}-c_{e_1}t_n+\tau_0,x_n).$$
Without loss of generality, we assume that there holds
$$u(x_n,t_n)>\Phi_{e_1}(x_{n1}-c_{e_1}t_n-\tau_0,x_n)\ \ \text{for all}\ n\in\mathbb{N}$$
or
$$u(x_n,t_n)<\Phi_{e_1}(x_{n1}-c_{e_1}t_n+\tau_0,x_n)\ \ \text{for all}\ n\in\mathbb{N}.$$
For the former case, since $\Phi_{e_1}(\xi,x)$ is strictly decreasing in $\xi$, we can conclude that
$$\frac{u(x_n,t_n)-\Phi_{e_1}(x_{n1}-c_{e_1}t_n,x_n)}{\Phi_{e_1}(x_{n1}-c_{e_1}t_n,x_n)}>\frac{\Phi_{e_1}(x_{n1}-c_{e_1} 
t_n-\tau_0,x_n)-\Phi_{e_1}(x_{n1}-c_{e_1}t_n,x_n)}{\Phi_{e_1}(x_{n1}-c_{e_1}t_n,x_n)}>0$$
for all $n\in\mathbb{N}$. That is to say, for all $n\in\mathbb{N}$,
$$\left\vert\frac{u(x_n,t_n)}{\Phi_{e_1}(x_{n1}-c_{e_1}t_n,x_n)}-1\right\vert>\frac{\Phi_{e_1}(x_{n1}-c_{e_1}t_n-\tau_0, 
x_n)}{\Phi_{e_1}(x_{n1}-c_{e_1}t_n,x_n)}-1.$$
Letting $n\rightarrow+\infty$ in the above inequality, by \eqref{jkl} we have that the left side of the inequality 
converges to $0$, however, by Lemma \ref{lem1} one obtains that the right side of the inequality converges to $e^{c_{e_1}
\tau_0}-1>0$, which leads to a contradiction.
		
For the latter case, through a similar process we obtain a contradiction. Therefore, we prove the validity of \eqref{bnm}.
		
By \eqref{bnm}, since $\Phi_{e_1}(\xi,x)$ is $L$-periodic with respect to $x$, we have that for any $t\in \mathbb{R}$,
$$\Phi_{e_1}(x_1-c_{e_1}t+\overline{x_n}_1-c_{e_1}s_n+\tau,x)\leq u_n(x,t)\leq\Phi_{e_1}(x_1-c_{e_1}t+\overline{x_n}_1-
c_{e_1}s_n-\tau,x)$$
for all $x\in\{x\in \overline{\Omega}-\{\overline{x_n}\}:x_1-c_{e_1}t+\overline{x_n}_1-c_{e_1}s_n\geq Q_\tau\}$. By 
taking $n\rightarrow+\infty$, we can derive that 
$$\Phi_{e_1}(x_1-c_{e_1}t+\tau,x)\leq U(x,t)\leq \Phi_{e_1}(x_1-c_{e_1}t-\tau,x)$$
for all $(x,t)\in\mathbb{R}^N\times\mathbb{R}$ satisfying $x_1-c_{e_1}t\geq Q_\tau$. Then, it can be inferred from Lemma 
\eqref{lem7} that
$$U(x,t)\equiv \Phi_{e_1}(x_1-c_{e_1}t,x)\ \ \text{for all}\ (x,t)\in\mathbb{R}^N\times\mathbb{R}.$$
However, by \eqref{ooo} and the definition of $u_n(x,t)$, we have
$$\left\vert\frac{u_n(\widehat{x_n},t_n-s_n)-\Phi_{e_1}(\xi_n,\widehat{x_n})}{\Phi_{e_1}^\beta(\xi_n,\widehat{x_n})} 
\right\vert\geq\epsilon_0 \ \ \text{for all}\ n\in \mathbb{N}.$$
Letting $n \rightarrow +\infty$, we have
$$\left\vert\frac{U(x^*,\frac{-\xi^*+x^*_1}{c_{e_1}})-\Phi_{e_1}(\xi^*,x^*)}{\Phi_{e_1}^\beta(\xi^*,x^*)}\right\vert\geq 
\epsilon_0 ,$$
since $U(x,t)\equiv \Phi_{e_1}(x_1-c_{e_1}t,x)\ \text{ for all }(x,t)\in\mathbb{R}^N\times\mathbb{R}$, we obtain a 
contradiction. The proof is complete.
\end{proof}
	
\section{Uniqueness of entire solutions}\label{sec5}

In this section, we first prove that the entire solution will converge to the pulsating front at infinity in the 
direction perpendicular to the propagation direction $e_1$ locally uniformly for $(x,t)\in\mathbb{R}^N\times\mathbb{R}$. 
Then, under the assumption \eqref{jkl}, by using the sliding method, we prove the uniqueness of the entire solution 
satisfying  the conditions of Theorem \ref{th2}.
	
\begin{lemma}\label{lem8}
Let the entire solution $u(x,t)$ be given by Theorem \ref{th2}. Then for any sequence $\{x'_n\}\subset \mathbb{R}^{N-1}$ 
with $\vert x'_n\vert\rightarrow +\infty$ as $n\rightarrow +\infty$, there holds
\begin{equation}\label{mzm}
u(x_1,x'+x'_n,t)-\Phi_{e_1}(x_1-c_{e_1}t,x_1,x'+x'_n)\rightarrow 0 \ \ \text{as}\ n\rightarrow +\infty
\end{equation}
locally uniformly in $(x,t)\in\mathbb{R}^N\times\mathbb{R}$.
\end{lemma}
	
\begin{proof}[\ \textbf{Proof}]
Firstly, we prove that for any sequence $\{x'_n\}\subset \{(k_2L_2,\cdots,k_NL_N):k_2,\cdots,k_N\in\mathbb{Z}\}$ with 
$\vert x'_n\vert\rightarrow +\infty$ as $n\rightarrow +\infty$ and $\{y_n\}\subset [0,L_2)\times\cdots \times[0,L_N)$ 
with $ y_n\rightarrow y^*\in[0,L_2]\times\cdots \times[0,L_N]$ as $n\rightarrow +\infty$, there holds
\begin{equation}\label{ww}
u(x_1,x'+x'_n+y_n,t)-\Phi_{e_1}(x_1-c_{e_1}t,x_1,x'+x'_n+y_n)\rightarrow 0 \ \ \text{as}\ n\rightarrow +\infty
\end{equation}
locally uniformly in $(x,t)\in\mathbb{R}^N\times\mathbb{R}$. Define
$$v_n(x,t)=u(x_1,x'+x'_n,t)\ \ \text{for all}\ x\in\Omega-\{(0,x'_n)\}\ \ \text{and}\ t \in \mathbb{R}.$$	
Since $K$ is a compact set, from the standard parabolic estimates one gets that $v_n(x,t)$ converges to a solution 
$v(x,t)$ of equation \eqref{f2} locally uniformly for $(x,t)\in\mathbb{R}^N\times\mathbb{R}$ as $n\rightarrow+\infty$. 
In fact, if we can prove that $v(x,t)\equiv\Phi_{e_1}(x_1-c_{e_1}t,x)$, then \eqref{ww} is satisfied, since $u(x_1,x'+
x'_n+y_n,t)\rightarrow v(x_1,x'+y^*,t)$ as $n\rightarrow +\infty$ and $\Phi_{e_1}(x_1-c_{e_1}t,x_1,x'+x'_n+y_n)\rightarrow
\Phi_{e_1}(x_1-c_{e_1}t,x'+y^*)$ as $n\rightarrow +\infty$.
		
Next we prove that $v(x,t)\equiv\Phi_{e_1}(\xi,x)$ for $(x,t)\in \mathbb{R}^N\times\mathbb{R}$. Since $u(x_1,x'+x'_n,t)
\rightarrow v(x,t)$ and $\Phi_{e_1}(\xi,x)$ is $L$-periodic in $x$, it follows from \eqref{uxt1} that
$$\lim_{t \to -\infty} \frac{v(x,t)-\Phi_{e_1}(\xi,x)}{\Phi_{e_1}^\beta(\xi,x)}=0 \ \ \text{uniformly for}\ x \in 
\mathbb{R}^N \ \ \text{and}\ \beta \in (0,\min\{\beta_0,\beta_1\}].$$
This yields that for any $\zeta>0$, there exists $T_\zeta<0$ such that for any $x\in\mathbb{R}^N$,
\begin{equation}\label{ee}
\begin{split}
&\ \Phi_{e_1}(x_1-c_{e_1}T_\zeta,x)-\zeta\Phi_{e_1}^\beta(x_1-c_{e_1}T_\zeta,x)\\
\leq &\ v(x,T_\zeta)\leq \Phi_{e_1}(x_1-c_{e_1}T_\zeta,x)+\zeta\Phi_{e_1}^\beta(x_1-c_{e_1}T_\zeta,x).
\end{split}
\end{equation}
Taking $0<\zeta<\frac{\theta_1}{2\Vert \delta_\beta(x)\Vert_{L^\infty(\mathbb{R}^N)}}$, similar to the proof of Lemma 
\ref{lem6}, it is straightforward to show that there exist $\mu>0$ and $\rho>0$ such that  the following functions
$$v^+(x,t)=\Phi_{e_1}(\xi_1(x,t),x)+\zeta\Phi_{e_1}^\beta(\xi_1(x,t),x)e^{-\mu t}$$
and
$$v^-(x,t)=\Phi_{e_1}(\xi_2(x,t),x)+\zeta\Phi_{e_1}^\beta(\xi_2(x,t),x)e^{-\mu t}$$
are the super-solution and sub-solution of equation \eqref{f2} for all $x\in \mathbb{R}^N$ and $t\geq 0$, where
$$\xi_1(x,t)=x_1-c_{e_1}(t+T_\zeta)-\rho\zeta(1-e^{-\mu t}),$$
$$\xi_2(x,t)=x_1-c_{e_1}(t+T_\zeta)+\rho\zeta(1-e^{-\mu t}).$$
Then, by \eqref{ee} we get that
$$v^-(x,0)\leq v(x,T_\zeta)\leq v^+(x,0).$$
It follows from the comparison principle that
$$v^-(x,t-T_\zeta)\leq v(x,t)\leq v^+(x,t-T_\zeta)\ \ \text{for all}\ x\in \mathbb{R}^N \ \text{ and }t\geq T_\zeta.$$
Letting $T_\zeta\rightarrow -\infty$, one concludes that
$$\Phi_{e_1}(x_1-c_{e_1}t+\rho\zeta,x)\leq v(x,t)\leq \Phi_{e_1}(x_1-c_{e_1}t-\rho\zeta,x)\ \ \text{for all}\ x\in \mathbb
{R}^N \ \text{ and }t\in \mathbb{R},$$
since $\zeta$ can be arbitrarily small and is independent of $\rho$, we conclude that $v(x,t)\equiv\Phi_{e_1}(x_1-c_{e_1}
t,x)$ for $(x,t)\in \mathbb{R}^N\times\mathbb{R}$. This further implies \eqref{mzm}. The proof is complete.
\end{proof}
	
\vspace{0.3cm}
\begin{proof}[\ \textbf{Proof of Theorem \bf{\ref{th6}}}]
Since $u(x,t)$ and $\widehat{u}(x,t)$ satisfy  the conditions of Theorem \ref{th2}, by Theorem \ref{th4} one infers that 
there exists $B>0$ such that
\begin{equation*}
\begin{split}
u(x,t),\widehat{u}(x,t)\leq \theta_1\ \ \text{for all}\ (x,t)\in\{(x,t)\in\overline{\Omega}\times\mathbb{R}:x_1-c_{e_1}t
\geq B\},\\
u(x,t),\widehat{u}(x,t)\geq 1-\theta_1\ \ \text{for all}\ (x,t)\in\{(x,t)\in\overline{\Omega}\times\mathbb{R}:x_1-c_{e_1}t
\leq -B\},
\end{split}
\end{equation*}
where $\theta_1$ is given by \eqref{eta1theta1}. Since $u(x,t)$ and $\widehat{u}(x,t)$ are transition fronts connecting 
$0$ and $1$, by \cite[Theorem 1.2]{jjj} one gets that there exists $0<\theta_2\leq \theta_1$ such that
$$\theta_2\leq u(x,t),\widehat{u}(x,t)\leq 1-\theta_2\ \ \text{for all}\ (x,t)\in\{(x,t)\in\overline{\Omega}\times\mathbb
{R}:\vert x_1-c_{e_1}t\vert \leq B\}.$$
Moreover, we infer from Theorem \ref{th4} that there exists a sufficiently large constant $t_0>0$ such that
\begin{equation*}
\begin{split}
\widehat{u}(x,t-t_0)\leq \theta_2\leq u(x,t)\ \ \text{for all}\ (x,t)\in\{(x,t)\in\overline{\Omega}\times\mathbb{R}:\vert 
x_1-c_{e_1}t\vert \leq B\},\\
\widehat{u}(x,t-t_0)\leq 1-\theta_2\leq u(x,t)\ \ \text{for all}\ (x,t)\in\{(x,t)\in\overline{\Omega}\times\mathbb{R}: 
\vert x_1-c_{e_1}(t-t_0)\vert \leq B\}.
\end{split}
\end{equation*}
Combining the above formulas, we have
\begin{equation*}
\widehat{u}(x,t-t_0)\leq u(x,t)\ \ \text{for all}\ (x,t)\in\{(x,t)\in\overline{\Omega}\times\mathbb{R}:-B-c_{e_1}t_0\leq 
x_1-c_{e_1}t \leq B\}.
\end{equation*}
Using \eqref{eta1theta1} and following the argument of  \cite[Lemma 4.2]{jjj}, we can apply the sliding method to prove 
that
\begin{equation}\label{nnn}
\widehat{u}(x,t-t_0)\leq u(x,t)
\end{equation}
for all $(x,t)\in\{(x,t)\in\overline{\Omega}\times\mathbb{R}: x_1-c_{e_1}t\leq-B-c_{e_1}t_0\ \text{ or } x_1-c_{e_1}t\geq 
 B\}$. In conclusion, we have that for any $t_1\geq t_0$,
$$\widehat{u}(x,t-t_1)\leq \widehat{u}(x,t-t_0)\leq u(x,t)\ \ \text{for all}\ (x,t)\in\overline{\Omega}\times\mathbb{R}.$$
		
Define
\begin{equation}\label{vvv}
t_*:=\inf\{t_2\in\mathbb{R}:\widehat{u}(x,t-t_2)\leq u(x,t)\ \ \text{for all}\ (x,t)\in\overline{\Omega}\times\mathbb
{R}\}.
\end{equation}
Note that $t_*\leq t_0$. Next we prove that $t_*\geq0$.	If $t_*<0$, we take a sequence $\{(x_n,t_n)\}\subset\overline
{\Omega}\times\mathbb{R}$ satisfying $x_n=x_*-nL$, $t_n\rightarrow -\infty$ and $x_{n1}-c_{e_1}t_n\rightarrow \xi_*\in
\mathbb{R}$ as $n\rightarrow +\infty$, where $x^*\in [0,L_1]\times\cdots\times[0,L_N]$ is a constant vector and $nL=(nL_1,
\cdots,nL_N)$. Then, by Theorem \ref{th2} we have
$$\lim_{n \to +\infty}[\widehat{u}(x_n,t_n-t_*)-u(x_n,t_n)]=\Phi_{e_1}(\xi_*+c_{e_1}t_*,x_*)-\Phi_{e_1}(\xi_*,x_*)>0,$$
which leads to a contradiction. 
		
We assume further that $t_*>0$. By the definition of $t_*$ in \eqref{vvv}, one gets that there are the following two 
cases may occur:
$$\inf\{u(x,t)-\widehat{u}(x,t-t_*):(x,t)\in\overline{\Omega}\times\mathbb{R}\ \ \text{satisfying}\ -B-c_{e_1}t_*\leq x_
1-c_{e_1}t \leq B\}>0$$
or
$$\inf\{u(x,t)-\widehat{u}(x,t-t_*):(x,t)\in\overline{\Omega}\times\mathbb{R}\ \ \text{satisfying}\ -B-c_{e_1}t_*\leq x_
1-c_{e_1}t \leq B\}=0.$$
		
For the former case, there exists a real number $0<t_{**}\leq t_*$ such that
$$\widehat{u}(x,t-(t_*-t_{**}))\leq u(x,t)\ \ \text{for all}\ (x,t)\in\Omega\times\mathbb{R}\ \ \text{satisfying}\ -B-
c_{e_1}t_*\leq x_1-c_{e_1}t \leq B .$$
Through a similar proof process to that of \eqref{nnn}, we can conclude that $\widehat{u}(x,t-(t_*-t_{**}))\leq u(x,t)$ 
for all $(x,t)\in\overline{\Omega}\times\mathbb{R}$, which leads to a contradiction.
		
For the latter case, we can find a sequence $\{(x_n,t_n)\}\subset\overline{\Omega}\times \mathbb{R}$ such that  $-B-c_
{e_1}t_*\leq (x_n)_1-c_{e_1}t_n \leq B$ for any $n\in \mathbb{N}$ and $u(x_n,t_n)-\widehat{u}(x_n,t_n-t_*)\rightarrow 0$ 
as $n\rightarrow +\infty$. Decompose $x_n$ into $\overline{x_n}+\widehat{x_n}$, where $\overline{x_n}\in \{(k_1L_1,\cdots,
k_NL_N):k_1,\cdots,k_N\in\mathbb{Z}\}$ and $\widehat{x_n}\in[0,L_1)\times\cdots \times[0,L_N)$. Without loss of 
generality, we assume that $x_{n1}-c_{e_1}t_n\rightarrow \xi_*\in[-B-c_{e_1}t_*,B]$ and $\widehat{x_n}\rightarrow x_*\in 
[0,L_1]\times\cdots\times [0,L_N]$ as $n\rightarrow +\infty$, meanwhile, we can assume that $t_n\rightarrow -\infty$, 
$t_n\rightarrow +\infty$ or $t_n\rightarrow t'\in\mathbb{R}$ as $n\rightarrow +\infty$.
		
If $t_n\rightarrow -\infty$ as $n\rightarrow +\infty$, then by Theorem \ref{th2} one has that
$$\lim_{n \to +\infty}[u(x_n,t_n)-\widehat{u}(x_n,t_n-t_*)]=\Phi_{e_1}(\xi_*,x_*)-\Phi_{e_1}(\xi_*+c_{e_1}t_*,x_*)>0,$$
which leads to a contradiction.
		
If $t_n\rightarrow +\infty$ as $n\rightarrow +\infty$, similarly by Theorem \ref{th5} we can get a contradiction.
		
If $t_n\rightarrow t'$ as $n\rightarrow +\infty$, then $x_{n1}\rightarrow y_1\in \mathbb{R}$ as $n\rightarrow +\infty$. 
In view of this, if $\vert x_n'\vert \rightarrow +\infty$ as $n\rightarrow +\infty$, then by Lemma \ref{lem8} we can get 
a contradiction. Therefore, we may assume without loss of generality that $x_n\rightarrow y(=(y_1,y'))\in\overline
{\Omega}$ as $n\rightarrow +\infty$. In conclusion, we get $u(y,t')=\widehat{u}(y,t'-t_*)$.
		
Let $w(x,t)=u(x,t)-\widehat{u}(x,t-t_*)$ for all $(x,t)\in \overline{\Omega}\times\mathbb{R}$. By \eqref{vvv} we have 
$w(x,t)\geq 0$ for all $(x,t)\in \overline{\Omega}\times\mathbb{R}$. In particular, it holds that $w(y,t')=0$. Since 
$u(x,t)$ and $\widehat{u}(x,t-t_*)$ are the solutions of problem \eqref{f1}, we can conclude that
\begin{equation*}
\begin{cases}
w_t=\Delta w+f(x,u(x,t))-f(x,\widehat{u}(x,t-t_*)),\ \ x\in\Omega,\ t\in\mathbb{R},\\
\frac{\partial u}{\partial n}=0, \ \ x\in\partial\Omega, \ t\in\mathbb{R}.
\end{cases}
\end{equation*}
By the mean value theorem, we can set $f(x,u(x,t))-f(x,\widehat{u}(x,t-t_*))=b(x,t)w$, where $b(x,t)$ is bounded since 
$\sup_{(x,u)\in\mathbb{R}^N\times\mathbb{R}}\vert f_u(x,u)\vert<+\infty$. Then, by using Hopf's lemma and the strong 
maximum principle we can derive that $w(x,t)\equiv 0$ for all $(x,t)\in \overline{\Omega}\times\mathbb{R}$. However, by 
Theorem \ref{th2} we can conclude a contradiction easily. In conclusion, we have $t_*=0$, which means that $w(x,t)=u(x,t)-
\widehat{u}(x,t)\geq 0$ for all $(x,t)\in \overline{\Omega}\times\mathbb{R}$.
		
By reversing the role of \(\widehat{u}(x,t)\) and \(u(x,t)\), one arrives at $\widehat{u}(x,t)-u(x,t)\geq 0$ for all 
$(x,t)\in \overline{\Omega}\times\mathbb{R}$. Then, $\widehat{u}(x,t)\equiv u(x,t)$ for all $(x,t)\in \overline{\Omega}
\times\mathbb{R}$. The proof is complete.
\end{proof}
	
\section*{Acknowledgments}
The third author's work was partially supported by NSF of China (12171120) and by the Fundamental Research Funds for the 
Central Universities.
	
\section*{Date availability statements}
We do not analyze or generate any datasets, because our work proceeds within a theoretical and mathematical approach.

\section*{Conflict of interest}
There is no conflict of interest to declare.


\begin{thebibliography}{99}
		
\footnotesize{	
	
\bibitem{AG} M. Alfaro, T. Giletti, Varying the direction of propagation in reaction-diffusion equations in periodic 
media, {\it Networks and Heterogeneous Media.} {\bf 11} (2016), no.  3, 369-393.

\bibitem{H3} H. Berestycki, F. Hamel, Front propagation in periodic excitable media, {\it Comm. Pure Appl. Math.} {\bf 
55} (2002), no.  8, 949-1032.
						
\bibitem{bh1} H. Berestycki, F. Hamel, Generalized travelling waves for reaction-diffusion equations, In: \emph{
Perspectives in Nonlinear Partial Differential Equations}. In honor of H. Brezis, Amer. Math. Soc., Contemp. Math. 
\textbf{446} (2007), 101-123.
			
\bibitem{jjj} H. Berestycki, F. Hamel, Generalized transition waves and their properties, {\it Comm. Pure Appl. Math.} 
{\bf 65} (2012), no.  5, 592-648.	
					
\bibitem{H1} H. Berestycki, F. Hamel, H. Matano, Bistable traveling waves around an obstacle, {\it Comm. Pure Appl. 
Math.} {\bf 62} (2009), no.  6, 729-788.
			
\bibitem{bhm} H. Berestycki, F. Hamel, H. Matano, Front propagation through a perforated wall，{\it Comm. Pure Appl. 
Math.}, https://doi.org/10.1002/cpa.70051.
			
\bibitem{bo} J. Bouhours,  Robustness for a Liouville type theorem in exterior domains,	{\it J. Dyn. Differ. Equ.} \ 
\textbf{27} (2015), 297-306.
			
\bibitem{bc} J. Brasseur, J. Coville, A counterexample to the Liouville property of some nonlocal problems, {\it  Ann. 
Inst. H. Poincar\'{e} C, Anal. Non Lin\'eaire} \ \textbf{37} (2020), 549-579.
			
\bibitem{bshv}J.  Brasseur, J. Coville, F. Hamel,  E. Valdinoci, Liouville type results for a nonlocal obstacle problem,
{\it Proc. Lond. Math. Soc.} \ \textbf{119} (2019), 291-328 .
			
\bibitem{Z1} Z.-H. Bu, J.-F. He, Qualitative properties of pulsating fronts for reaction-advection-diffusion equations in 
periodic excitable media, {\it Nonlinear Anal. Real World Appl.} {\bf 63} (2022), Paper No. 103418, 19 pp.
			
\bibitem{buchong21} W. Ding, F. Hamel, X.-Q. Zhao, Transition fronts for periodic bistable reaction-diffusion equations, 
{\it Calc. Var. Partial Differ. Equ.} {\bf 54} (2015), 2517-2551.
			
\bibitem{buchong3} W. Ding, F. Hamel, X.-Q. Zhao, Bistable pulsating fronts for reaction-diffusion equations in a 
periodic habitat, {\it Indiana Univ. Math. J.} {\bf 66} (2017), 1189-1265. 
			
\bibitem{buchong2} A. Ducrot, A multi-dimensional bistable nonlinear diffusion equation in a periodic medium, {\it Math. 
Ann.} {\bf 366} (2016), 783-818. 
			
\bibitem{buchong11}	A. Ducrot, T. Giletti, H. Matano, Existence and convergence to a propagating terrace in 
one-dimensional reaction-diffusion equations, {\it Trans. Am. Math. Soc.} {\bf 366} (2014), 5541-5566.
			
\bibitem{fzhao} J. Fang, X.-Q. Zhao, Bistable traveling waves for monotone semiflows with applications, {\it J. Eur. 
Math. Soc.} {\bf 17} (2015), no.  9, 2243-2288.
			
\bibitem{P1} P.  C. Fife, J.  B. Mcleod, The approach of solutions of nonlinear diffusion equations to travelling front 
solutions, {\it Arch. Rational Mech. Anal.} {\bf 65} (1977), no.4, 335-361.
			
\bibitem{D1} D. Gilbarg, N. S. Trudinger, {\it Elliptic Partial Differential Equations of Second Order}, Springer, New 
York, 2001.
			
\bibitem{buchong5} T. Giletti, L. Rossi, Pulsating solutions for multidimensional bistable and multistable equations, 
{\it Math. Ann. } {\bf 378} (2020), no.  3-4, 1555-1611.
			
\bibitem{ghs}H. Guo, F. Hamel, W.-J. Sheng, On the mean speed of bistable transition fronts in unbounded domains, {\it J. 
Math. Pures Appl.}, {\bf 136} (2020), 92-157.
			
\bibitem{H2} H. Guo, H. Monobe, $V$-shaped fronts around an obstacle, {\it Math. Ann.} {\bf 379} (2021), no. 1-2, 661-689.
			
\bibitem{buchong8} F. Hamel, Qualitative properties of monostable pulsating fronts: exponential decay and monotonicity, 
{\it J. Math. Pures Appl.} {\bf 89} (2008), 355-399.
			
\bibitem{hamel}F. Hamel, Bistable transition fronts in $\mathbb R^N$, \textit{Adv. Math.}, \textbf{289} (2016), 279-344.
			
\bibitem{buchong9} F. Hamel, L.  J. Roques, Uniqueness and stability properties of monostable pulsating fronts, {\it J. 
Eur. Math. Soc.} {\bf 13} (2011), no.  2, 345-390. 
			
\bibitem{hr}F. Hamel, L. Rossi, Admissible speeds of transition fronts for non-autonomous monostable equations, 
\emph{SIAM J. Math. Anal.}, \textbf{47} (2015), 3342-3392.
			
\bibitem{hr1}F. Hamel, L. Rossi, Transition fronts for the Fisher-KPP equation, \emph{Trans. Amer. Math. Soc.}, 
\textbf{368} (2016), 8675-8713.
			
{\bibitem{HCWY} B.-S. Han, M.-X. Chang, H.-L. Wei, Y. Yang, Curved fronts for a Belousov-Zhabotinskii system in exterior 
domains, {\it J. Differential Equations}, \textbf{416} (2025), 1660-1695.}
			
\bibitem{A1} A. Hoffman, H.  J. Hupkes, E.  S. Van  Vleck, Entire solutions for bistable lattice differential equations 
with obstacles, {\it Mem. Amer. Math. Soc.} {\bf 250} (2017), no.  1188, v+119 pp.
			
{ \bibitem{JBB} F.-J. Jia, X. Bao, W.-J. Bo, V-shaped front-like solutions of the buffered bistable system in exterior 
domains, {\it J. Differential Equations}, \textbf{431} (2025), 40 pp.}
			
\bibitem{fff} F.-J. Jia, Y.-C. Ma, X.-X. Bao, G.-H. Guo, Pulsating front-like solutions of spatially periodic combustion  
reaction-diffusion equations in exterior domains, {\it Calc. Var.  Partial Differential Equations} {\bf 65} (2026), no. 
6, Paper No. 194.
			
\bibitem{F1} F.-J. Jia, W.-J. Sheng, Z.-C. Wang, Pulsating fronts of spatially periodic bistable reaction-diffusion 
equations around an obstacle, {\it J. Nonlinear Sci.} {\bf 34} (2024), no.  1, Paper No. 4, 37 pp.

\bibitem{F2} F.-J. Jia, Z.-C. Wang, Z.-H. Bu, Propagation phenomena of time-periodic combustion reaction-diffusion 
equations in exterior domains, {\it Discrete Contin. Dyn. Syst.} {\bf 45} (2025), no.  4, 1187-1247.
			
\bibitem{L1} L.-L. Li, Time-periodic planar fronts around an obstacle, {\it J. Nonlinear Sci.} {\bf 31} (2021), no.  6, 
Paper No. 90, 21 pp.
			
\bibitem{J4} Y. Lyu, H. Guo, Z.-C. Wang, On traveling fronts of combustion equations in spatially periodic media, {\it J. 
Dynam. Differential Equations.} {\bf 37} (2025), no.  4, 2869-2929.
			
\bibitem{buchong12}	J. Nolen, L. Ryzhik, Traveling waves in a one-dimensional heterogeneous medium, {\it Ann. Inst. H. 
Poincaré Analyse Non Linéaire} {\bf 26} (2009), 1021-1047.
			
\bibitem{S1} S.-X. Qiao, W.-T. Li, J.-W. Sun, Propagation phenomena for nonlocal dispersal equations in exterior domains, 
{\it J. Dynam. Differential Equations.} {\bf 35} (2023), no.  2, 1099-1131.
			
\bibitem{shen1}W. Shen, Traveling waves in diffusive random media, {\it J. Dyn. Differ. Equ. } \textbf{16} (2004), 
1011-1060.
			
\bibitem{shen2} W. Shen, Traveling waves in time dependent bistable equations, {\it Differ. Integral Equ.} \textbf{19} 
(2006), 241-278.
			
\bibitem{shen3} W. Shen, Existence, uniqueness, and stability of generalized traveling waves in time dependent monostable 
equations, {\it J. Dyn. Differ. Equ.} \textbf{23} (2011), 1-44.
			
\bibitem{shen4}W. Shen, Z. Shen, Stability, uniqueness and recurrence of generalized traveling waves in time 
heterogeneous media of ignition type, {\it Trans. Am. Math. Soc.} \textbf{369} (2017), 2573-2613.
			
\bibitem{slww} W.-J. Sheng,  L. Li,  Z.-C.Wang,  M. Wang, Transition fronts of time periodic bistable reaction-diffusion 
equations around an obstacle, {\it J. Anal. Math.} \textbf{155} (2025), 165-233.
				
\bibitem{sal} Y.-J. Shi, L.-L. Li, Stability of pulsating fronts for bistable reaction-diffusion equations in spatially 
periodic media, {\it J. Math. Anal. Appl.} {\bf 539} (2024), 128516.
			
\bibitem{N1} N. Shigesada, K. Kawasaki, E. Teramoto, Traveling periodic waves in heterogeneous environments, {\it 
Theoret. Population Biol.} {\bf 30} (1986), no.  1, 143-160.

\bibitem{J1} J.-X. Xin, Existence and uniqueness of travelling waves in a reaction-diffusion equation with combustion 
nonlinearity, {\it Indiana Univ. Math. J.} {\bf 40} (1991), no.  3, 985-1008.
			
\bibitem{J2} J.-X. Xin, Existence and stability of traveling waves in periodic media governed by a bistable nonlinearity, 
{\it J. Dynam. Differential Equations.} {\bf 3} (1991), no.  4, 541-573.
			
\bibitem{J3} J.-X. Xin, Existence of planar flame fronts in convective-diffusive periodic media, {\it Arch. Rational 
Mech. Anal.} {\bf 121} (1992), no.  3, 205-233.
			
\bibitem{buchong24} J.-X. Xin, Existence and nonexistence of traveling waves and reaction-diffusion front propagation in 
periodic media, {\it J. Stat. Phys.} {\bf 73} (1993), 893-926.
			
\bibitem{buchong22}	J.-X. Xin, J.-Y. Zhu, Quenching and propagation of bistable reaction-diffusion fronts in 
multidimensional periodic media, {\it Phys. D} {\bf 81} (1995),  94-110.
			
\bibitem{ys} Y.-Y. Yan, W.-J. Sheng, Transition fronts of combustion reaction-diffusion equations around an obstacle, 
{\it Calc. Var. Partial Differential Equations}, \textbf{63} (2024), 63 pp.
			
\bibitem{ys1} Y.-Y. Yan, W.-J. Sheng, Transition fronts of monotone bistable reaction-diffusion systems around an 
obstacle,  {\it Comm. Partial Differential Equations}, https://doi.org/10.1080/03605302.2026.2689044.
			
\bibitem{ys2} Y.-Y. Yan, W.-J. Sheng, V-shaped transition fronts of monotone bistable reaction-diffusion systems in 
exterior domains, arXiv:2603.22896.
			
}
\end{thebibliography}
\end{document}